\documentclass[11pt]{amsart}

\usepackage[utf8]{inputenc}
\usepackage{a4wide}
\usepackage{color}
\usepackage{amsfonts,amsmath,amsthm,amssymb}
\usepackage{verbatim}
\usepackage{epsfig}
\usepackage{graphics}
\usepackage{mathtools}
\usepackage{array}
\usepackage{enumerate}
\usepackage[ruled]{algorithm}
\usepackage{algorithmic}
\usepackage{bbm}
\usepackage{caption}
\usepackage{subcaption}
\usepackage[all, knot, cmtip]{xy} 
\usepackage{wrapfig}
\usepackage{fixltx2e}
\usepackage{bm}

\newcommand{\ndN}{\mathbb{N}}

\newcommand{\ndR}{\mathbb{R}}

\newcommand{\ndK}{\mathbb{K}}

\renewcommand{\Pr}[1]{\mathbb{P}\left(#1\right)}
\newcommand{\Ex}[1]{\mathbb{E}[#1]}

\newcommand{\one}{\mathbbm{1}}
\newcommand{\w}{w}

\newcommand{\sca}{\kappa}

\newcommand \isom{
{\xrightarrow{\,\smash{\raisebox{-0.65ex}{\ensuremath{\scriptstyle\sim}}}\,}}
}

\newcommand{\rhoc}{\rho}
\newcommand{\rhob}{R}
\newcommand{\lambdac}{\lambda}

\newcommand{\cC}{\mathcal{C}}
\newcommand{\cK}{\mathcal{K}}

\newcommand{\cF}{\mathcal{F}}
\newcommand{\cB}{\mathcal{B}}
\newcommand{\cS}{\mathcal{S}}
\newcommand{\cG}{\mathcal{G}}

\newcommand{\cR}{\mathcal{R}}
\newcommand{\cA}{\mathcal{A}}
\newcommand{\cH}{\mathcal{H}}
\newcommand{\cD}{\mathcal{D}}
\newcommand{\cX}{\mathcal{X}}
\newcommand{\cP}{\mathcal{P}}
\newcommand{\cY}{\mathcal{Y}}

\newcommand{\mA}{\mathsf{A}}
\newcommand{\mC}{\mathsf{C}}
\newcommand{\mB}{\mathsf{B}}
\newcommand{\mG}{\mathsf{G}}
\newcommand{\mF}{\mathsf{F}}
\newcommand{\mH}{\mathsf{H}}
\newcommand{\mT}{\mathsf{T}}
\newcommand{\mS}{\mathsf{S}}
\newcommand{\mP}{\mathsf{P}}

\newcommand{\me}{\mathsf{e}}

\newcommand{\CRT}{\mathcal{T}_\me}

\newcommand{\cE}{\mathcal{E}}
\newcommand{\cBa}{\cD}
\newcommand{\cBra}{\cD^\bullet}
\newcommand{\Ba}{D}
\newcommand{\Bra}{D^\bullet}

\newcommand{\plh}{%
  {\ooalign{$\phantom{0}$\cr\hidewidth$\scriptstyle\times$\cr}}%
}

\newcommand{\He}{\text{H}}
\newcommand{\Di}{\text{D}}

\newcommand{\shp}{\mathsf{sp}}

\newcommand{\lhb}{\mathsf{lb}}
\newcommand{\Forb}{\mathsf{Forb}}
\newcommand{\spa}{\mathsf{span}}

\newcommand{\cCb}{\mathcal{C}^\bullet}

\newcommand{\cTb}{\mathcal{A}}

\newcommand{\eqdist}{\,{\buildrel (d) \over =}\,}

\newcommand{\convdislong}{\,{\buildrel (d) \over \longrightarrow}\,}

\newcommand{\disto}{\mathsf{dis}}

\newcommand{\Set}{\textsc{SET}}
\newcommand{\Seq}{\textsc{SEQ}}

\newtheorem{theorem}{Theorem}[section]

\newtheorem{corollary}[theorem]{Corollary}
\newtheorem{proposition}[theorem]{Proposition}
\newtheorem{lemma}[theorem]{Lemma}

\newtheorem{definition}[theorem]{Definition}

\numberwithin{equation}{section}

\newcommand{\keyword}[1]{\textbf{#1}~}
\newcommand{\IF}{\keyword{if}}

\newcommand{\THEN}{\keyword{then} }
\newcommand{\ELSE}{\keyword{else}}
\newcommand{\ENDIF}{\keyword{endif} }
\newcommand{\RETURN}{\keyword{return}}

\newcommand{\FOR}{\keyword{for}}

\newcommand{\EACH}{\keyword{each}}

\newcommand{\SPENDFOR}{\textbf{end for}~ }

\newcommand{\TO}{\keyword{to}}

\newcommand{\Pois}[1]{{{\mathsf{Pois}}\left(#1\right)}}
\newcommand{\Bern}[1]{{{\mathsf{Bern}}\left(#1\right)}}

\title{Scaling Limits of Random Graphs from Subcritical Classes}

\author{Konstantinos Panagiotou}
\address[Konstantinos Panagiotou]{Institute of Mathematics, Ludwig-Maximilians-Universit\"at, Theresienstr.\ 39, 80333 Munich, Germany}
\email{kpanagio@math.lmu.de}

\author{Benedikt Stufler}
\address[Benedikt Stufler]{Institute of Mathematics, Ludwig-Maximilians-Universit\"at, Theresienstr.\ 39, 80333 Munich, Germany}
\email{stufler@math.lmu.de}

\author{Kerstin Weller}
\address[Kerstin Weller]{Institute of Theoretical Computer Science, ETH, 8092 Zurich, Switzerland}
\email{kerstin.weller@inf.ethz.ch}

\date{\today}

\begin{document}

\begin{abstract}
We study the uniform random graph $\mC_n$ with $n$ vertices drawn from a subcritical class of connected graphs. Our main result is that the rescaled graph $\mC_n / \sqrt{n}$ converges to the Brownian Continuum Random Tree $\CRT$ multiplied by a constant scaling factor that depends on the class under consideration. In addition, we provide subgaussian tail bounds for the diameter $\Di(\mC_n)$ and height $\He(\mC_n^\bullet)$ of the rooted random graph $\mC_n^\bullet$.  We give analytic expressions for the scaling factor of several classes, including for example the prominent class of outerplanar graphs. Our methods also enable us to study first passage percolation on $\mC_n$, where we show the convergence to $\CRT$ under an appropriate rescaling.
\end{abstract}

\maketitle

\section{Introduction}

Let $G$ be a connected graph with vertex set $V(G)$ and edge set $E(G)$. We can associate in a natural way a metric space $(V(G), d_G)$ to $G$, where $d_G(u,v)$ is the number of edges on a shortest path that contains $u$ and $v$ in $G$. In this work we study the case where $G$ is a random graph, and we consider several properties of the associated metric space as the number of vertices of~$G$ becomes large.

In the series of seminal papers~\cite{MR1085326,MR1166406,MR1207226} Aldous studied the fundamental case of $G$ being a critical Galton-Watson random tree with $n$ vertices, where the offspring distribution has finite nonzero variance. Among other results, he showed that the asymptotic properties of the associated metric space admit an \emph{universal} description: they can be depicted, up to an appropriate rescaling, in terms of “continuous trees” whose archetype is the so-called \emph{Brownian Continuum Random Tree} (CRT for short). Since Aldous's pioneering work, the CRT has been identified as the limiting object of many different classes of discrete structures, in particular trees, see e.g.\ Haas and Miermont \cite{MR3050512} and references therein, and planar maps, see e.g.\ Albenque and Marckert \cite{MR2438817}, Bettinelli \cite{Be}, Caraceni \cite{Ca}, Curien, Haas and Kortchemski \cite{CuHa} and Janson and Stefansson \cite{JaSi}. 

Although the aforementioned papers identify the CRT as the universal limiting object in various settings, much less is known about the scaling limit of random graphs from complex graph classes. In this paper we study in a unified way the asymptotic distribution of distances in random graphs from so-called \emph{subcritical classes}, where, informally, a class is called subcritical if for a typical graph with $n$ vertices the largest block (i.e.~inclusion maximal 2-connected subgraph) has $O(\log n)$ vertices. Random graphs from such classes have been the object of intense research in the last years, see e.g.~\cite{MR3184197,MR2534261,MR2873207,MR2675698}, especially because of their close connection to the class of planar graphs. Prominent examples of classes  that are subcritical are outerplanar and series-parallel graphs. However, with the notable exception of~\cite{MR3184197}, most research on such random graphs has focused on \emph{additive} parameters, like the number of vertices of a given degree; the fine study of \emph{global} properties, like the distribution of the distances, poses a significant challenge.

In the present paper we study the random graph $\mC_n$ drawn uniformly from the set of connected graphs with $n$ vertices of a subcritical class $\cC$. Our first main result is Theorem~\ref{te:maintheorem}, which shows that, up to an appropriate rescaling, the associated metric space converges in distribution to a multiple of the CRT. Postponing the introduction of the appropriate notation to later sections (see the outline), our main result asserts that there is a constant $s = s(\cC) > 0$  such that
\[
	(V(\mC_n), s\, n^{-1/2}\, d_{\mC_n}) \convdislong \CRT,
\]
where $\CRT$ is the CRT and convergence is with respect to the Gromov-Hausdorff metric. In particular, this establishes that the CRT is the universal scaling limit for random graphs from subcritical clases, and it proves (in a strong form) a conjecture by Drmota and Noy~\cite{MR3184197}. The proof of Theorem~\ref{te:maintheorem}, see Section~\ref{sec:convergence}, gives a natural combinatorial characterization of the ``scaling'' constant $s$. Our methods are based on the algebraic concept of~$\cR$-enriched trees; more specifically, we use a size-biased enriched tree in order to study a coupling of $\mC_n$ with an appropriate conditioned critical Galton-Watson tree $\mT_n$ sharing the same vertex set. Our main step in the proof establishes that with high probability for any two vertices $u, v$ in $\mC_n$  the distance $d_{\mC_n}(u,v)$ is concentrated around~$d_{\mT_n}(u,v)$ multiplied by a constant factor $\sca \ge 1$ that depends \emph{only} on~$\cC$, and which is, very roughly speaking, the average length of a shortest path between two random distinct vertices in a random block of~$\mC_n$. Thus, the constant $s$ turns out to be the product of two quantities: the constant involved in the scaling limit of $\mT_n$, and the reciprocal of $\sca$. In Section~\ref{sec:examples} we exploit this characterization of~$s$ and compute its value for several important classes, including outerplanar graphs.

As a consequence of our main result we obtain the following statements, see Corollary~\ref{cor:hddistr}. The \emph{diameter} $\Di(G)$ of a graph $G$ is defined as $\max_{u,v \in V(G)} d_G(u,v)$, and the \emph{height} $\He(G^\bullet)$ of a pointed graph $G^\bullet=(G,r)$, which is $G$ rooted at a vertex $r$, is $\max_{v \in V(G)} d_G(r,v)$. Then the limit distribution for the diameter of~$\mC_n$ and the height of $\mC_n^\bullet$ satifsy for $x > 0$, as $n\to\infty$
\begin{align*}
\Pr{\Di(\mC_n) > {s^{-1}n^{1/2}x}} &\to \sum_{k=1}^\infty (k^2-1)\Big(\frac{2}{3}k^4x^4 -4k^2x^2 +2\Big)\exp(-k^2x^2/2), \\
\Pr{\He(\mC_n^\bullet) > {s^{-1}n^{1/2}x}} &\to 2 \sum_{k=1}^\infty (4k^2x^2-1)\exp(-2k^2x^2). 
\end{align*}
Apart from the convergence in distribution, we also show sharp tail bounds for the diameter and the height, see Theorem~\ref{te:tailbound}. In particular, we show that there are constants $C,c>0$ such that  \emph{for all} $n$ and $x \ge 0$
\[
\Pr{\Di(\mC_n) \ge x} \le C \exp(-cx^2/n) \quad \text{and} \quad \Pr{\He(\mC_n^\bullet) \ge x} \le C \exp(-cx^2/n).
\]
A similar result was shown for critical Galton-Watson random trees by Addario-Berry, Devroye and Janson \cite{MR3077536}, and our proof of these bounds builds on the methods in that paper. From this we deduce that all moments of the rescaled height and diameter converge as well. In particular, we obtain the universal and remarkable asymptotic behaviour  
\begin{align*}
\Ex{\Di(\mC_n)} \sim \frac{2^{3/2}}{3s} \sqrt{\pi n} \sim \frac{4}{3} \Ex{\He(\mC_n^\bullet)}.
\end{align*}
This improves the previously best known bounds $c_1 \sqrt{n} \le \Ex{\Di(\mC_n)} \le c_2 \sqrt{n \log{n}}$ given in~\cite{MR3184197}. The higher moments can also be determined and are depicted in Sections~\ref{sec:convergence} and~\ref{sec:heightcrt}.

In addition to the previous results, we demonstrate that our proof strategy is powerful enough to enable us to study the far more general setting of \emph{first passage percolation}: suppose that the edges of $\mC_n$ are equipped with independent random ``lengths'', drawn from a distribution that has exponential moments, and let the distance of two vertices $u,v$ be the minimum sum of those lengths along a path that contains both $u$ and $v$. Our last main results shows that again, up to an appropriate rescaling, the associated metric space converges to a multiple of the CRT; see Section~\ref{sec:extensions} for the details.

\subsection*{Outline}
Section~\ref{sec:crt} fixes some basic notation and summarizes several results related to Galton-Watson random trees and the Continuum Random Tree (CRT). In particular, Subsection~\ref{sec:heightcrt} states the distribution and expressions for arbitrarily high moments of the height and diameter of the CRT -- to our knowledge these results are scattered accross several papers and we provide a concise presentation. Section~\ref{prelim} is devoted to the definition of combinatorial species, $\cR$-enriched trees and subcritical graph classes. In this part of the paper we collect some general and relevant properties of the these objects -- many of them were already known in special cases, and we put them in a broader context. In Section~\ref{sec:pointedenr} we describe a construction of a novel powerful object called the \emph{size-biased random enriched tree} that will allow us to study systematically the distribution of distances in random graphs from subcritical graph classes. Subsequently, in Section~\ref{sec:convergence} we show our main result: the convergence of the rescaled random graphs towards a multiple of the CRT. Section~\ref{sec:tailbounds} complements this result by proving subgaussian tail-bounds for the height and diameter. Section~\ref{sec:extensions} is devoted to several extensions of our results, in particular first passage percolation. The paper closes with many examples, including the prominent class of outerplanar graphs.

\section{Galton-Watson Trees and the CRT}
\label{sec:crt}

We briefly summarize required notions and results related to the Brownian Continuum Random Tree (CRT) and refer the reader to Aldous \cite{MR1207226} and Le Gall \cite{MR2203728} for a thorough treatment.

\subsection{Graphs and (Plane) Trees}
\label{prelim1}
All graphs considered in this paper are undirected and may not contain multiple edges or loops. That is, a graph $G$ consists of a non-empty set $V(G)$ of vertices and a set $E(G)$ of edges that are two-element subsets of $V(G)$. If $|V(G)| \in \mathbb{N}$ we say that $|G| := |V(G)|$ is the {\em size} of $G$. Following Diestel \cite{MR2744811}, we recall and fix basic definitions and notations. Two vertices $v, w \in V(G)$ are said to be {\em adjacent} if $\{v,w\} \in E(G)$. We will often write $vw$ instead of $\{v,w\}$. A \emph{path} $P$ is a graph such that
\[
V(P)=\{v_0, \ldots, v_\ell\}, \qquad E(P)=\{ v_0v_1, \ldots, v_{\ell-1}v_\ell\}
\]
with the $v_i$ being distinct. The number $\ell = \ell(P)$ of edges of a path is its {\em length}. We say $P$ {\em connects} its endvertices $v_0$ and $v_\ell$ and we often write $P=v_0v_1 \ldots v_\ell$. If $P$ has length at least two we call the graph $C_\ell = P + v_0 v_\ell$ obtained by adding the edge $v_0v_\ell$ a \emph{cycle}. The \emph{complete graph} with $n$ vertices in which each pair of distinct vertices is adjacent is denoted by $K_n$.

We say the graph $G$ is \emph{connected} if any two vertices $u, v \in V(G)$ are connected by a path in $G$. The length of a shortest path connecting the vertices $u$ and $v$ is called the {\em distance} of $u$ and $v$ and it is denoted by $d_G(u,v)$. Clearly $d_G$ is a metric on the vertex set $V(G)$. A graph~$G$ together with a distinguished vertex $v \in V(G)$ is called a {\em rooted} graph with root-vertex $v$. The {\em height} $h(w)$ of a vertex $w \in V(G)$ is its distance from the root. The \emph{height} $H(G)$ is the maximum height of the vertices in $G$. A {\em tree} $T$ is a non-empty connected graph without cycles. Any two vertices of a tree are connected by a unique path. If $T$ is rooted, then the vertices $w' \in V(T)$ that are adjacent to a vertex $w$ and have height $h(w) + 1$ form the {\em offspring} set of the vertex $w$. Its cardinality is the {\em outdegree} $d^+(w)$ of  $w$.

The {\em Ulam-Harris tree} is an infinite rooted tree with vertex set $\cup_{n \ge 0} \ndN^n$ consisting of finite sequences of natural numbers. The empty string $\emptyset$ is the root and the offspring of any vertex $v$ is given by the concatenations $v1, v2, v3, \ldots$. In particular, the labelling of the vertices induces a linear order on each offspring set. A {\em plane tree} is defined as a subtree of the Ulam-Harris tree that contains the root.

\subsection{Galton-Watson Trees}
\label{sec:gwt}
Throughout this section we fix an integer-valued random variable $\xi \ge 0$. By abuse of language we will often not distinguish between $\xi$ and its distribution. A \emph{$\xi$-Galton-Watson} tree $\mT$ is the family tree of a Galton-Watson branching process with offspring distribution $\xi$, interpreted as a (possibly infinite) plane tree. If $\Pr{\xi = 1} < 1$, then $\mT$ is almost surely finite if and only if $\Ex{\xi} \le 1$. If $\Ex{\xi} = 1$ we call $\mT$ \emph{critical}. Let $\mathsf{supp}(\xi) = \{k \mid \Pr{\xi = k} > 0\}$ denote the {\em support} of $\xi$ and define the {\em span} $d = \spa(\xi)$ as the greatest common divisor of $\{k - \ell \mid k, \ell \in \mathsf{supp}(\xi)\}$.  If $\mT$ is finite, then 
\[
 |\mT| = 1 + \sum_{v \in V(\mT)} d^+_{\mT}(v) \equiv 1 \mod d.
\]
Conversely, if $\xi$ is not almost surely positive, then $\Pr{|\mT| = n} > 0$ for all large enough $n\in\mathbb{N}$ satisfying $n \equiv 1 \mod d$. We will need the following standard result for the probability that a critical Galton-Watson tree has size $n$.
\begin{lemma}[{\cite[p.\ 105]{MR865130}}]
\label{le:cycap}
\label{le:gwtsize}
Suppose that the distribution $\xi$ has expected value one and finite nonzero variance $0 < \sigma^2 < \infty$. Let $(\xi_i)_{i \in \ndN}$ be an infinite family of i.i.d.\ copies of~$\xi$. Then the probability that the $\xi$-Galton-Watson tree $\mT$ has size $n \equiv 1 \mod d$ with $d = \spa(\xi)$ satisfies
\[
\Pr{|\mT| = n} = n^{-1} \Pr{ \sum_{i=1}^n \xi_i = n-1} \sim{} \frac{d}{\sqrt{2 \pi \sigma^2}} n^{-3/2} \quad \text{as $n \to \infty$}.
\]
\end{lemma}
We also state some results given in~\cite{MR3077536,MR2908619} that will be useful in our arguments. As in the previous lemma, suppose that $\xi$ satisfies $\Ex{\xi}=1$ and has finite nonzero variance. Let $n \equiv 1 \mod d$ be sufficiently large such that $\Pr{\mT = n} > 0$ and let $\mT_n$ denote the Galton-Watson tree conditioned on having size $n$. By \cite[Thm.\ 1.2]{MR3077536} and \cite[p.\ 6]{MR3077536} there are constants $C_1,c_1, C_2, c_2 > 0$ such that the height $\He(\mT_n)$ satisfies the following tail bounds for all $n$ and $h \ge 0$:
\begin{align}
\label{eq:gwttail}
\Pr{\He(\mT_n) \le h} \le C_1 \exp(-c_1 (n-2)/h^2),
\qquad
\Pr{\He(\mT_n) \ge h} \le C_2 \exp(-c_2 h^2/ n).
\end{align}

\subsection{Convergence of Rescaled Conditioned Galton-Watson trees}
\label{sec:convcrt}

\emph{Throughout this and the following subsections  in Section~\ref{sec:crt} we write $\mT$ for a critical $\xi$-Galton-Watson tree having finite nonzero variance $\sigma^2$. Moreover, $n$ will always denote an integer satisfying $n \equiv 1 \mod\spa(\xi)$ and is assumed to be large enough such that the conditioned tree $\mT_n$ having exactly $n$ vertices is well-defined.
}

Given a plane tree~$T$ of size $n$ we  consider its canonical \emph{depth-first search} walk $(v_i)_{0 \le i \le 2(n-1)}$ that starts at the root and always traverses the leftmost unused edge first. That is, $v_0$ is the root of $T$ and given $v_0, \ldots, v_i$ walk if possible to the leftmost unvisited son of $v_i$. If $v_i$ has no sons or all sons have already been visited, then try to walk to the parent of~$v_i$. If this is not possible either, being only the case when~$v_i$ is the root of $T$ and all other vertices have already been visited, then terminate the walk. The corresponding heights $c(i) := d(\text{root of } T, v_i)$ define the search-depth function $c$ of the tree $T$. The \emph{contour function} $C: [0,2(n-1)] \to \ndR_+$ is defined by $C(i) = c(i)$ for all integers $0 \le i \le 2(n-1)$ with linear interpolation between these values, see Figure~\ref{fi:contour} for an example.
\begin{figure}[ht]
	\centering
	\begin{minipage}{0.8\textwidth}
  		\centering
  		\includegraphics[width=0.8\textwidth, height=2cm]{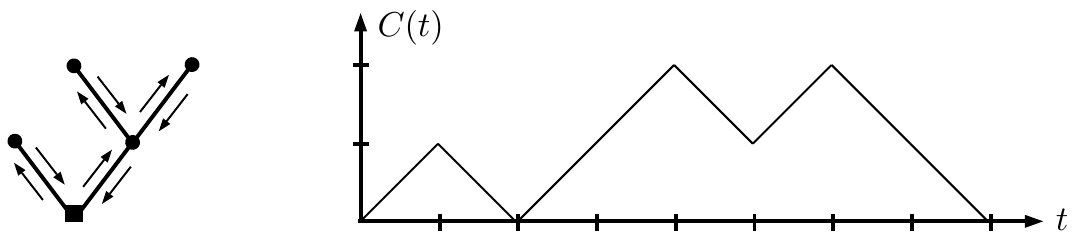}
  		\caption{The contour function of a plane tree.}
  		\label{fi:contour}
	\end{minipage}
\end{figure}
It can be shown that after a suitable rescaling the contour process of $\mT_n$ converges to a normalized Brownian excursion.
\begin{theorem}[{\cite[Thm 6.1]{MR2603061}}]
\label{theo:gwtconv}
Let $\mT_n$ be a critical $\xi$-Galton-Watson tree conditional on having $n$ vertices, where $\xi$ has finite non-zero variance $\sigma^2$. Let $C_n$ denote the contour function of $\mT_n$. Then
\[
\left (\frac{\sigma}{2 \sqrt{n}} C_n\big(t2(n-1)\big) \right)_{0 \le t \le 1} \convdislong ~~\me
\] in $\mathcal{C}([0,1], \ndR_+)$, where $\me = (\me_t)_{0 \le t \le 1}$ denotes the Brownian excursion of duration one.
\end{theorem}
This result is due to Aldous {\cite[Thm.\ 23]{MR1207226}}, who stated it for aperiodic offspring distributions. See also Duquesne \cite{MR1964956}
and Kortchemski \cite{MR3185928} for further extensions. Theorem~\ref{theo:gwtconv} can be formulated as a convergence of random trees with respect to the Gromov-Hausdorff metric. We introduce the required notions following Le Gall and Miermont \cite{MR3025391}.

A continuous function $g: [0,1] \to [0, \infty)$ with $g(0) = g(1) = 0$ induces a pseudo-metric on the interval $[0,1]$ by
\[
d_g(u,v) = g(u) + g(v) - 2 \inf_{u \le s \le v} g(s)
\]
for $u \le v$. This defines a metric on the quotient $T_g = [0,1]{/\hspace{-1.1mm}\sim}$, where $u \sim v$ if and only if $d(u,v) = 0$; we denote the corresponding metric space by $(T_g,d_g)$ and call the equivalence class $r_0(T_g)$ of the origin its \emph{root}.

Given a metric space $(E, d)$ the set of compact subsets is again a metric space with respect to the \emph{Hausdorff metric}
\[
\delta_{\text{H}}(K,K') = \inf\{\epsilon > 0 \mid K \subset U_{\epsilon}(K'), K' \subset U_{\epsilon}(K)\},
\]
where $U_{\epsilon}(K) = \{x \in E \mid d(x,K) \le \epsilon\}$. The sets $\ndK~(\ndK^\bullet)$ of isometry classes of (pointed) compact metric spaces, where a pointed  space is a triple $(E,d,r)$ where $(E,d)$ is a metric space and $r\in E$ is a distingished element, are Polish spaces with respect to the \emph{(pointed) Gromov-Hausdorff metric}
\[
\begin{split}
d_{\text{GH}}((E_1,d_1),(E_2,d_2)) & = \inf_{\varphi_1, \varphi_2} \left\{ \epsilon > 0 \mid \delta_{\text{H}}(\varphi_1(E_1), \varphi_2(E_2)) \right\}, \\
d_{\text{GH}}((E_1,d_1, r_1),(E_2,d_2,r_2)) & = \inf_{\varphi_1, \varphi_2} \left\{ \epsilon > 0 \mid \max\big\{\delta_{\text{H}}(\varphi_1(E_1), \varphi_2(E_2)), d_E(\varphi_1(r_1), \varphi_2(r_2)\right\}\big\},
\end{split}
\]
where the infimum is in both cases taken over all isometric embeddings $\varphi_1: E_1 \to E$ and $\varphi_2: E_2 \to E$ into a common metric space $(E, d_E)$ \cite[Thm.\ 3.5]{MR3025391}. We will make use of the following characterisation of the Gromov-Hausdorff metric. Given two compact metric spaces $(E_1, d_1)$ and $(E_2, d_2)$ a \emph{correspondence} between them is a subset $R \subset E_1 \times E_2$ such that any point $x \in E_1$ corresponds to at least one point $y \in E_2$ and vice versa. The \emph{distortion} of $R$ is 
\[
\disto(R) = \sup\{|d_1(x_1,y_1) - d_2(x_2,y_2)| \mid (x_1, y_1),(x_2,y_2) \in R \}.
\]
\begin{proposition}[{\cite[Prop.\ 3.6]{MR3025391}}]
\label{pro:distortion}
Given two pointed compact metric spaces $(E_1,d_1,r_1)$ and $(E_2,d_2,r_2)$ we have that
\[
	2d_{GH}((E_1,d_1, r_1),(E_2,d_2,r_2)) = \inf_R \disto(R),
\]
where $R$ ranges over all correspondences between $E_1$ and $E_2$ such that $r_1,r_2$ correspond to each other.
\end{proposition}
\noindent An important consequence is that the mapping 
\[
(\{g \in \cC([0,1], \ndR_+) \mid g(0)=g(1)=0 \}, \lVert \cdot \lVert_{\infty}) \to (\ndK, d_{GH}), \qquad g \mapsto T_g
\] is Lipschitz-continous \cite[Cor.\ 3.7]{MR3025391}.
\begin{definition}
The random metric space $(\CRT, d_\me, r_0(\CRT))$ coded by the Brownian excursion of duration one $\me$ is called the Brownian continuum random tree (CRT).
\end{definition}
\noindent
We may view $\mT_n$ as a random pointed metric space $(V(\mT_n), d_{\mT_n}, \emptyset) \in \ndK$. This space is close to the real tree encoded by its contour function, hence Theorem \ref{theo:gwtconv} implies convergence with respect to the Gromov-Hausdorff metric, see \cite[p.\ 740]{MR2603061}.
\begin{theorem}
\label{te:gwtconv}
Let $\mT_n$ be a critical $\xi$-Galton-Watson tree conditional on having $n$ vertices, where $\xi$ has finite non-zero variance $\sigma^2$. As $n$ tends to infinity,  $\mT_n$ with edges rescaled to length $\frac{\sigma}{2 \sqrt{n}}$ converges in distribution to the CRT, that is
\begin{equation}
\label{eq:convFullDetail}
	\left(V(\mT_n), \frac{\sigma}{2 \sqrt{n}} d_{\mT_n}, \emptyset \right) \convdislong (\CRT, d_\me, r_0(\CRT))
\end{equation}
in the metric space $(\ndK^\bullet, d_{\text{GH}})$.
\end{theorem}
This invariance principle is due to Aldous \cite{MR1207226} and there exist various extensions. See for example Duquesne \cite{MR1964956}, Duquesne and Le Gall \cite{MR2147221}, Haas and Miermont \cite{MR3050512}. Adopting common terminology, we will often write in the sequel
\[
	\frac{\sigma}{2 \sqrt{n}} \mT_n \convdislong \CRT
\]
instead of~\eqref{eq:convFullDetail}.

\subsection{Height and Diameter of the CRT}
\label{sec:heightcrt}

Clearly the height $\He(\mT_n)$ and Diameter $\Di(\mT_n)$ of the tree $\mT_n$ may be recovered from its contour function $C$. By the results from Section~\ref{sec:convcrt} it follows that
\begin{align}
\label {eq:distheightconv}
\frac{\sigma}{2 \sqrt{n}}\He(\mT_n) &\convdislong  \He(\CRT) \eqdist \sup_{0 \le t \le 1}\me(t) \\
\label {eq:distdiamconv}
\frac{\sigma}{2 \sqrt{n}}\Di(\mT_n) &\convdislong  \Di(\CRT) \eqdist \sup_{0 \le t_1 \le t_2 \le 1}(\me(t_1) + \me(t_2) - 2 \inf_{t_1 \le t \le t_2} \me(t)).
\end{align}
Since $\Di(\mT_n) \le 2 \He(\mT_n)$ the tail bound (\ref{eq:gwttail}) implies that all moments in (\ref{eq:distheightconv}) and (\ref{eq:distdiamconv}) converge.
It is well-known that $\He(\CRT)/\sqrt{2}$ follows a \emph{Theta distribution}, that is 
\begin{align}
\label{eq:disheight}
\Pr{\He(\CRT) > x} =  2 \sum_{k=1}^\infty (4k^2x^2-1)\exp(-2k^2x^2).
\end{align}
for all $x>0$. The $k$th moment of the height is given by
\begin{align}
\label{eq:exheight}
\Ex{\He(\CRT)} = \sqrt{\pi / 2} \quad \text{and} \quad \Ex{\He(\CRT)^k} = 2^{-k/2} k(k-1) \Gamma(k/2) \zeta(k) \quad \text{for $k \ge 2$}.
\end{align}
This follows from standard results on the Brownian excursion, see e.g.\ Chung \cite{MR0467948} and Biane, Pitman and Yor \cite{MR1848256}, or by calculating directly the limit distribution of extremal parameters of a class of trees that converges to the CRT (see e.g.\ R\'enyi and Szekeres \cite{MR0219440}). The distribution of the diameter is given by
\begin{align}
\label{eq:disdiam}
\Pr{\Di(\CRT) > x} =
\sum_{k=1}^\infty (k^2-1)\Big(\frac{2}{3}k^4x^4 -4k^2x^2 +2\Big)\exp(-k^2x^2/2).
\end{align}
This expression may be obtained (by tedious calculations) from results of Szekeres \cite{MR731595}, who proved the existence of a limit distribution for the diameter of rescaled random unordered labelled trees. The moments of this distribution were calculated for example in Broutin and Flajolet \cite{MR2956055} and are given by
\begin{align}
\label{eq:exdiam}
\Ex{\Di(\CRT)} = \frac{4}{3}\sqrt{\pi/2}, \quad \Ex{\Di(\CRT)^2} = \frac{2}{3}\left(1 + \frac{\pi^2}{3}\right), \quad \Ex{\Di(\CRT)^3} = 2 \sqrt{2\pi}, \\
\label{eq:exdiam2}
\Ex{\Di(\CRT)^k} = \frac{2^{k/2}}{3} k(k-1)(k-3) \Gamma(k/2)(\zeta(k-2) - \zeta(k)) \quad \text{for $k \ge 4$}.
\end{align}

\section{Combinatorial Species, $\cR$-enriched Trees and Subcritical Graph Classes}
\label{prelim}
We recall parts of the theory of combinatorial species and Boltzmann samplers to the extend required in this paper. A reader who is already familiar with the framework of subcritical graph classes may  skip some parts of this section. However, we stress the importance of the representation of connected graphs as enriched trees in Subsection~\ref{sec:enriched} and the coupling of random graphs with a Galton-Watson tree in Subsection \ref{sec:subcri}. Moreover, several intermediate lemmas that we state were already shown in previous papers, albeit under stronger assumptions.

\subsection{Combinatorial Species}
\label{prelim2}
The framework of combinatorial species allows for a unified treatment of a wide range of combinatorial objects. We give only a concise introduction and refer to Joyal \cite{MR633783} and Bergeron, Labelle and Leroux \cite{MR1629341} for a detailed discussion. The essentially equivalent language of {\em combinatorial classes} was developed by Flajolet and Sedgewick~\cite{MR2483235}.

A {\em combinatorial species} may be defined as a functor $\cF$ that maps any finite set $U$ of {\em labels} to a finite set $\cF[U]$ of {\em $\cF$-objects} and any bijection $\sigma: U \to V$ of finite sets to its (bijective) {\em transport function} $\cF[\sigma]:\cF[U]\to\cF[V]$ {\em along} $\sigma$, such that composition of maps and the identity are preserved. We say that a species $\cG$ is a {\em subspecies} of $\cF$ and write $\cG \subset \cF$ if $\cG[U] \subset \cF[U]$ for all finite sets $U$ and $\cG[\sigma] = \cF[\sigma]|_{U}$ for all bijections $\sigma: U \to V$.
Given two species $\cF$ and $\cG$, an {\em isomorphism} $\alpha: \cF \, \isom\,  \cG$ from $\cF$ to $\cG$ is a family of bijections $\alpha = (\alpha_U: \cF[U] \to \cG[U])_{U}$ where $U$ ranges over all finite sets, such that $\cG[\sigma] \alpha_U = \alpha_V \cF[\sigma]$ for all bijective maps $\sigma: U \to V$. The species $\cF$ and $\cG$ are {\em isomorphic} if there exists and isomorphism from one to the other. This is denoted by $\cF \simeq \cG$ or, by abuse of notation, just $\cF = \cG$. An element $\gamma_U \in \cF[U]$ has size $|\gamma_U| := |U|$ and two $\cF$-objects $\gamma_U$ and $\gamma_V$ are termed isomorphic if there is a bijection $\sigma: U \to V$ such that $\cF[\sigma](\gamma_U) = \gamma_V$. We will often just write $\sigma.\gamma_U = \gamma_V$ instead, if there is no risk of confusion. An isomorphism class of $\cF$-structures is called an \emph{unlabeled} $\cF$-object.

We will mostly be interested in subspecies of the species of finite simple graphs and use basic species such as the species of linear orders $\Seq$ or the $\Set$-species defined by $\Set[U] = \{U\}$ for all $U$. Moreover let $0$ denote the empty species, $1$ the species with a single object of size $0$ and $\cX$ the species with a single object of size $1$.

We set $[n] := \{1, \ldots, n\}$ and $\cF_n := \cF[n] = \cF[\{1, \ldots, n\}]$ for all $n \in\mathbb{N}_0$. By abuse of notation we will often let $\cF$ also refer to the set $\cup_n \cF_n$. The \emph{exponential generating series} of a combinatorial species $\cF$ is defined by
$
 F(x) = \sum_{n \ge 0} |\cF_n| x^n/n!
$.
In general, this is a formal power series that may have radius of convergence zero. If the series $F(x)$ has positive radius of convergence, we say that $\cF$ is an \emph{analytic species} and $F(x)$ is its \emph{exponential generating function}. For any power series $f(x)$ we let $[x^n]f(x)$ denote the coefficient of $x^n$.

\subsection{Operations on Species}
\label{prelim3}
The framework of combinatorial species offers a large variety of constructions that create new species from others. In the following let $\cF$, $(\cF_i)_{i \in \ndN}$ and $\cG$ denote species and $U$ an arbitrary finite set.
The {\em sum} $\cF + \cG$ is defined by the disjoint union
\[
(\cF + \cG)[U] = \cF[U] \sqcup \cG[U].
\] 
More generally, the infinite sum $(\sum_i \cF_i)$ may be defined by
$
(\sum\nolimits_i \cF_i)[U] = \bigsqcup\nolimits_i \cF_i[U]
$
if the right hand side is finite for all finte sets $U$. The {\em product} $\cF \cdot \cG$ is defined by the disjoint union
\[
(\cF \cdot \cG)[U] = \bigsqcup_{\mathclap{\substack{ (U_1, U_2) \\ U_1 \cap U_2 = \emptyset, U_1 \cup U_2 = U}}} \cF[U_1] \times \cG[U_2] \\
\]
with componentwise transport. Thus, $n$-sized objects of the product are pairs of $\cF$-objects and $\cG$-objects whose sizes add up to $n$. If the species $\cG$ has no objects of size zero, we can form the {\em substitution} $\cF \circ \cG$ by
\[
(\cF \circ \cG)[U] = \bigsqcup_{\mathclap{\substack{ \pi \text{ partition of } U}}} \cF[\pi] \times \prod_{Q \in \pi} \cG[Q]. \\
\]
An object of the substition may be interpreted as an $\cF$-object whose labels are substituted by $\cG$-objects. The transport along a bijection $\sigma$ is defined by applying the induced map $\overline{\sigma}: \pi \to \overline{\pi}=\{\sigma(Q) \mid Q \in \pi \}$ of partitions to the $\cF$-object and the restriced maps $\sigma|_Q$ with $Q \in \pi$ to their corresponding $\cG$-objects. We will often write $\cF(\cG)$ instead of $\cF \circ \cG$. The {\em rooted} or {\em pointed} $\cF$-species is given by
\[
\cF^{\bullet}[U] = U \times \cF[U] \\
\]
with componentwise transport. That is, a pointed object is formed by distinguishing a label, named the \emph{root} of the object, and any transport function is required to preserve roots. The {\em derived} species $\cF'$ is defined by
\[
\cF'[U] = \cF[U \cup \{*_U\}]
\]
with  $*_U$ referring to an arbitrary fixed element not contained in the set~$U$. (For example, we could take $*_U = U$.) The transport along a bijective map $\sigma: U \to V$ is done by applying the canonically extended bijection $\sigma': U \sqcup \{*_U\} \to V \sqcup \{*_V\}$ with $\sigma'(*_U)=*_V$ to the object. Derivation and pointing are related by an isomorphism $\cF^\bullet \simeq X \cdot \cF'$. Note that $\cF'^\bullet$ and ${\cF^\bullet}'$ are in general different species, since a $\cF'^\bullet$-object may be rooted only at a non-$*$-label.

Explicit formulas for the exponential generating series of these constructions are summarized in Table~\ref{tb:egs}. The notation is quite suggestive: up to (canonical) isomorphism, each operation considered in this section is associative. The sum and product are commutative operations and satisfy a distributive law, i.e.
\begin{equation}
\label{eq:distrLaw}
\cF \cdot (\cG_1 + \cG_2) \simeq \cF \cdot \cG_1 + \cF \cdot \cG_2.
\end{equation}
The operation of deriving a species is additive and satisfies a product rule and a chain rule, analogous to the derivative in calculus:
\begin{equation}
\label{eq:chainRule}
(\cF \cdot \cG)' \simeq \cF' \cdot \cG + \cF \cdot \cG' \quad \text{and} \quad \cF(\cG)' \simeq \cF'(\cG) \cdot \cG'.
\end{equation}
Recall that for the chain rule to apply we have to require $\cG[\emptyset] = \emptyset$, since otherwise $\cF(\cG)$ is not defined. A thorough discussion of these facts is beyond the scope of this introduction. We refer the inclined reader to Joyal \cite{MR633783} and Bergeron, Labelle and Leroux \cite{MR1629341}.

{
\small
\renewcommand{\arraystretch}{1.2}
\begin{table}
	\begin{center}
	\begin{tabular}{lll}
    \begin{tabular}{ | l | l |}
     \hline
     $\sum_i \cF_i$ & $\sum_i F_i(x)$ \\
     $\cF \cdot \cG$ & $F(x) G(x)$ \\
     $\cF \circ \cG$ & $F(G(x))$ \\
     $\cF^\bullet$ & $x \frac{d}{dx} F(x)$ \\
     $\cF'$ & $\frac{d}{dx} F(x)$ \\
     \hline
    \end{tabular} & \qquad \qquad\qquad &
    \begin{tabular}{ | l | l |}
     \hline
     $\Set$ & $\exp(x)$ \\
     $\Seq$ & $1/(1-x)$ \\
     $0$ & $0$ \\
     $1$ & $1$ \\
     $\cX$ & $x$ \\
     \hline
    \end{tabular}
    \end{tabular}
	\end{center}
	\caption{Relation between combinatorial constructions and generating series.}
	\label{tb:egs}
\end{table}
}

\subsection{Combinatorial Specifications}
\label{prelim4}
In this section we briefly recall Joyal's implicit species theorem that allows us to define combinatorial species up to unique isomorphism and construct recursive samplers that draw objects of a species randomly (see Section~\ref{sec:Boltzmann} below). In order to state the theorem we need to introduce the concept of {\em multisort species}. As it is sufficient for our applications, we restrict ourselves to the 2-sort case.

A {\em $2$-sort species} $\cH$ is a functor that maps any pair $U= (U_1, U_2)$ of finite sets to a finite set $\cH[U] = \cH[U_1, U_2]$ and any pair $\sigma = (\sigma_1, \sigma_2)$ of bijections $\sigma_i: U_i \to V_i$ to a bijection $\cH[\sigma]: \cH[U] \to \cH[V]$ in such a way, that identity maps and composition of maps are preserved.  The operations of sum, product and composition extend naturally to the multisort-context. Let $\cH$ and $\cK$ be 2-sort species and $U = (U_1,U_2)$ a pair of finite sets. The {\em sum} is defined by
\[
(\cH + \cK)[U] = \cH[U] \sqcup \cK[U].
\]
We write $U = V + W$ if $U_i = V_i \cup W_i$ and $V_i \cap W_i = \emptyset$ for all $i$. The {\em product} is defined by
\[
(\cH \cdot \cK)[U] = \bigsqcup_{V + W = U} \cH[V] \times \cK[W].
\]
The {\em partial derivatives} are given by
\[
\partial_1 \cH[U] = H[U_1 \cup \{*_{U_1}\}, U_2] \quad \text{and} \quad \partial_2 \cH[U] = H[U_1, U_2 \cup \{*_{U_2}\}].
\] 
In order state Joyal's implicit species theorem we also require the substitution operation for multisort species; this will allow us to  define species ``recursively'' up to (canonical) isomorphism. Let $\cF_1$ and $\cF_2$ be (1-sort) species and $M$ a finite set. A structure of the {\em composition} $\cH(\cF_1, \cF_2)$ over the set $M$ is a quadrupel $(\pi, \chi, \alpha, \beta)$ such that: 
\begin{enumerate}
\item $\pi$ is partition of the set $M$.
\item $\chi: \pi \to \{1, 2\}$ is a function assigning to each class a sort.
\item $\alpha$ a function that assigns to each class $Q \in \pi$ a $\cF_{\chi(Q)}$ object $\alpha(Q) \in \cF_{\chi(Q)}[Q]$.
\item $\beta$ a $\cH$-structure over the pair $(\chi^{-1}(1), \chi^{-1}(2))$.
\end{enumerate}
This construction is {\em functorial}: any pair of isomorphisms $\alpha_1$, $\alpha_2$ with $\alpha_i: \cF_i \, \isom \, \cG_i$ {\em induces} an isomorphism $\cH[\alpha_1, \alpha_2]: \cH(\cF_1, \cF_2) \, \isom \, \cH(\cG_1, \cG_2)$.

Let $\cH$ be a $2$-sort species and recall that $\cX$ denotes the species with a unique object of size one. A solution of the system $\cY = \cH(\cX, \cY)$ is pair $(\cA, \alpha)$ of a species $\cA$ with $\cA[0] = 0$ and an isomorphism $\alpha: \cA  \,\isom\, \cH(\cX, \cA)$. An isomorphism of two solutions $(\cA, \alpha)$ and $(\cB, \beta)$ is an isomorphism of species $u:\cA \,\isom\, \cB$ such that the following diagram commutes:
\[
  			\xymatrix{ \cA \ar[d]^{u} \ar[r]^-{\alpha} &\cH(\cX,\cA)\ar[d]^{\cH(\text{id},u)}\\
  					   \cB \ar[r]^-{\beta} 		    &\cH(\cX,\cB)}
\]
We may now state Joyal's implicit species theorem.
\begin{theorem}[\cite{MR633783}, Théorème 6]
\label{te:implicitspecies}
Let $\cH$ be a 2-sort species satisfying $\cH(0,0)=0$. If $(\partial_2 \cH)(0,0)=0$, then the system $\cY = \cH(\cX, \cY)$ has up to isomorphism only one solution. Moreover, between any two given solutions there is exactly one isomorphism.
\end{theorem}
\noindent We say that an isomorphism $\cF \simeq \cH(\cX, \cF)$ is a {\em combinatorial specification} for the species $\cF$ if the 2-sort species $\cH$ satisfies the requirements of Theorem~\ref{te:implicitspecies}.

\subsection{Block-Stable Graph Classes}
\label{prelim5}
Any graph may be decomposed into its {\em connected components}, i.e. its maximal connected subgraphs. These connected components allow a {\em block-decomposition} which we recall in the following. Let $C$ be a connected graph. If removing a vertex $v$ (and deleting all adjacent edges) disconnects the graph, we say that $v$ is a {\em cutvertex} of $C$. The graph $C$ is {\em 2-connected}, if it has size at least three and no cutvertices. 

A {\em block} of an arbitrary graph $G$ is a maximal connected subgraph $B \subset G$ that does not have a cutvertex (of itself). It is well-known, see for example \cite{MR2744811}, that any block is either $2$-connected or an edge or a single isolated point. Moreover, the intersection of two blocks is either empty or a cutvertex of a connected component of $G$. If $G$ is connected, then the bipartite graph whose vertices are the blocks and the cutvertices of $G$ and whose edges are pairs $\{v, B\}$ with $v \in B$ is a tree and called the {\em block-tree} of~$G$.

Let $\cG$ denote a subspecies of the species of graphs, $\cC \subset \cG$ the subspecies of connected graphs in $\cG$ and $\cB \subset \cC$ the subspecies of all graphs in $\cC$, that are $2$-connected or consist of only two vertices joined by an edge.
We say that $\cG$ or $\cC$ is a {\em block-stable} class of graphs, if $\cB \ne 0$ and $G \in \cG$ if and only if every block of $G$ belongs to $\cB$ or is a single isolated vertex. 
Block-stable classes satisfy the following combinatorial specifications that can be found for example in Joyal \cite{MR633783}, Bergeron, Labelle and Leroux \cite{MR1629341} and Harary and Palmer \cite{MR0357214}:
\begin{align}
	\label{eq:blockstability_classes}
\cG \simeq \Set \circ \cC \qquad \text{and} \qquad \cC^\bullet \simeq \cX \cdot (\Set \circ \cB' \circ \cC^\bullet).
\end{align}
The first correspondence expresses the fact that we may form any graph on a given vertex set $U$ by partitioning $U$ and constructing a connected graph on each partition class. The specification for rooted connected graphs, illustrated in Figure~\ref{fi:decompblockstable}, is based on the construction of the block-tree. The idea is to interpret $\cB' \circ \cC^\bullet$-objects as graphs by connecting the roots of the $\cC^\bullet$ objects on the partition classes and the $*$-vertex with edges according to the $\cB'$-object on the partition. An object of $\Set \circ (B' \circ \cC^\bullet)$ can then be interpreted as a graph by identifying the $*$-vertices of the $B' \circ \cC^\bullet$ objects. This construction is compatible with graph isomorphisms, hence $\cC' \simeq \Set \circ \cB' \circ \cC^\bullet$ and the second specification in \eqref{eq:blockstability_classes} follows. By the rules for computing the generating series of species we obtain the equations
\begin{align}
\label{eq:blockstability_series}
G(x) = \exp(C(x)) \qquad \text{and} \qquad C^{\bullet}(x) &= x \exp(B'( C^{\bullet}(x))).
\end{align}

\begin{figure}[ht]
	\centering
	\begin{minipage}{1.0\textwidth}
  		\centering
  		\includegraphics[width=0.66\textwidth]{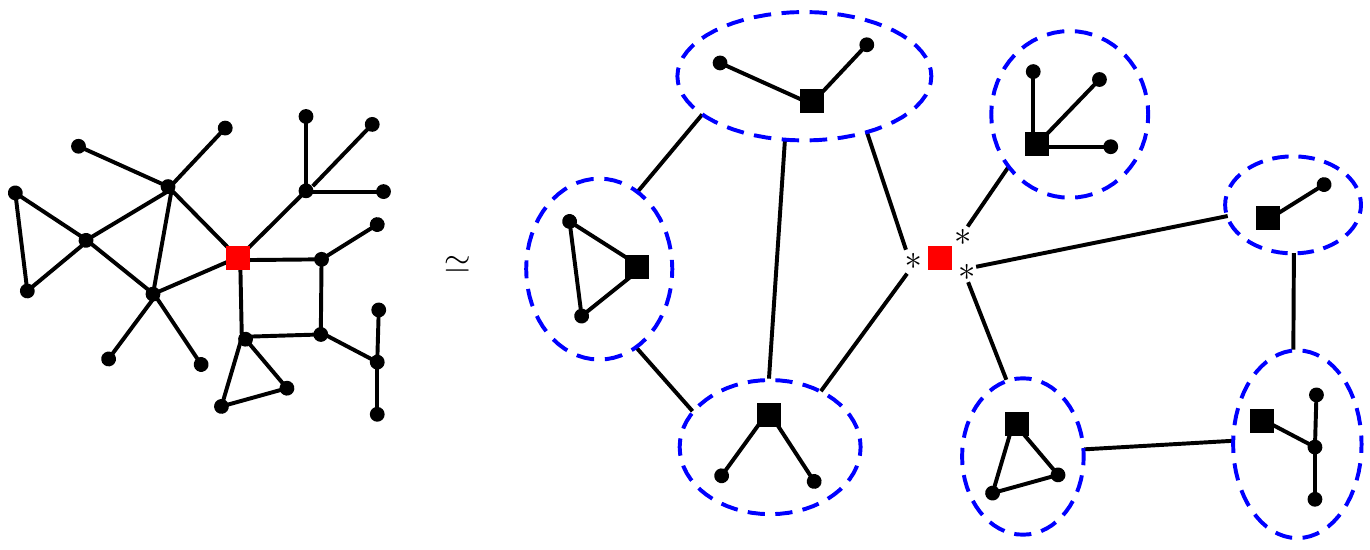}
  		\caption{Decomposition of a rooted graph from $\cC^\bullet$ into a $\cX \cdot (\Set \circ \cB' \circ \cC^\bullet)$ structure. Labels are omitted and the roots are marked with squares.}
  		\label{fi:decompblockstable}
	\end{minipage}
\end{figure}

The following lemma was given in Panagiotou and Steger \cite{MR2675698} and Drmota et al.~\cite{MR2873207} under some minor additional assumptions.
\begin{lemma}
\label{le:cfinite}
Let $\cC$ be a block-stable class of connected graphs, $\cB \ne 0$ its subclass of all graphs that are 2-connected or a single edge. Then the exponential generating series $C(z)$ has radius of convergence $\rhoc < \infty$ and the sums $y := C^\bullet(\rhoc)$ and $\lambdac := B'(y)$ are finite and satisfy
\begin{align}
\label{eq:blockconst}
y = \rhoc \exp(\lambdac).
\end{align}
\end{lemma}
\begin{proof}
It suffices to consider the case $\rhoc > 0$. By assumption we have $\cB \ne 0$ and hence there is a $k\in\mathbb{N}$ such that $[z^k]B'(z) \ne 0$. Thus, by~\eqref{eq:blockstability_series} we have, say, $C^\bullet(z) = c z C^\bullet(z)^{2k} + R(z)$  for some constant $c>0$ and $R(z)$ a power series in $z$ with nonnegative coefficients. This implies $\lim_{x \uparrow \rhoc} C^\bullet(x) < \infty$ and thus $\rhoc$ and $C^\bullet(\rhoc)$ are both finite. The coefficients of all power series involved in~\eqref{eq:blockstability_series} are nonnegative, and so  it follows that $y = \rhoc \exp(\lambda)$ and thus $\lambda < \infty$.
\end{proof}

We will only be interested in the case where $\cC$ is analytic. The following observation (made for example also in \cite{MR3184197}) shows that this is equivalent to requiring that $\cB$ is analytic. We include a short proof for completeness.
\begin{proposition}
\label{pro:analytic}
Let $\cC$ be a block-stable class of connected graphs, $\cB \ne 0$ its subclass of all graphs that are 2-connected or a single edge. Then $\cC$ is analytic if and only if $\cB$ is analytic.
\end{proposition}
\begin{proof}
By nonnegativity of coefficients we see easily that $\rhoc > 0$ implies that $\cB$ is analytic. Conversely, suppose that $B(z)$ has positive radius of convergence $\rhob > 0$.  By the inverse function theorem, the block-stability equation $f(z) = z \exp(B'(f(z)))$ has an analytic solution whose expansion at the point $0$ agrees with the series $C^\bullet(z)$ by Lagrange's inversion formula. Hence $\cC$ is an analytic class.
\end{proof}

\subsection{$\cR$-enriched Trees} \label{sec:enriched} The class $\cTb$ of rooted trees\footnote{ {\em Arborescence} is the French word for rooted tree, hence the notation $\cA$.} is known to satisfy the decomposition 
\[\cTb \simeq \cX \cdot \Set(\cTb).\]
This is easy to see: in order to from a rooted tree on a given set of vertices, we choose a root vertex $v$, partition the remaining the vertices, endow each partition class with a structure of a rooted tree and connect the vertex $v$ with their roots.  More generally, given a species $\cR$ the class $\cA_\cR$ of $\cR$-enriched trees is defined by the combinatorial specification 
\[
\cA_\cR \simeq \cX \cdot \cR(\cA_\cR).
\]
In other words, a $\cR$-enriched tree is a rooted tree such that the offspring set of any vertex is endowed with a $\cR$-structure. Natural examples are labeled ordered trees, which are $\Seq$-enriched trees, and plane trees, which are unlabeled ordered trees. Ordered and unordered tree families defined by restrictions on the allowed outdegree of internal vertices also fit in this framework. $\cR$-enriched trees were introduced by Labelle \cite{MR642392} in order to provide a combinatorial proof of Lagrange Inversion. They have applications in various fields of mathematics, see for example \cite{MR3149696,MR682626,MR2673891}.
 
The combinatorial specification \eqref{eq:blockstability_classes} together with Theorem~\ref{te:implicitspecies} allows us to identify a block-stable graph class $\cCb$ with the class $\cR$-enriched trees where $\cR = \Set(\cB')$, that is, rooted trees from $\cTb$ where the offspring set of each vertex is partitioned into nonempty sets and each of these sets carries a $\cB'$-structure. Compare with Figure \ref{fi:decomptree}.

\begin{figure}[ht]
	\centering
	\begin{minipage}{1.0\textwidth}
  		\centering
  		\includegraphics[width=1.0\textwidth]{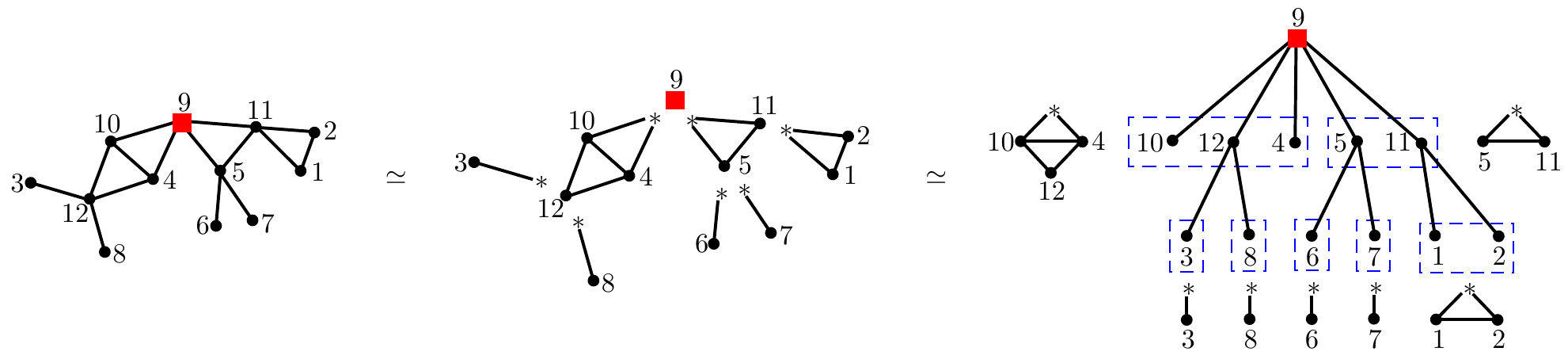}
  		\caption{Correspondence of the classes $\cCb$ and $\Set(\cB')$-enriched trees. }
  		\label{fi:decomptree}
	\end{minipage}
\end{figure}

\begin{corollary}
\label{co:encor}
Let $\cC$ be a block-stable class of connected graphs, $\cB \ne 0$ its subclass of all graphs that are 2-connected or a single edge. Then there is a unique isomorphism between~$\cCb$ and the class~$\cA_{\Set \circ \cB'}$ of pairs $(T, \alpha)$ with~$T \in \cTb$ and~$\alpha$ a function that assigns to each~$v \in V(T)$ a (possibly empty) set~$\alpha(v) \in (\Set \circ \cB')[M_v]$ of derived blocks whose vertex sets partition the offspring set $M_v$ of $v$. 
\end{corollary}
\begin{proof}
By the isomorphism given in \eqref{eq:blockstability_classes} the classes $\cA_{\Set \circ \cB'}$ and $\cC^\bullet$ are both solutions of the system $\cY = \cH(\cX, \cY)$ with $\cH(\cX, \cY) = \cX \cdot \Set \circ \cB' \circ \cY$. Joyal's Implicit Species Theorem~\ref{te:implicitspecies} yields that there is a unique isomorphism between any two solutions. 
\end{proof}

\subsection{Boltzmann Samplers} \label{sec:Boltzmann} Boltzmann samplers provide a method of  generating efficiently random discrete combinatorial objects. They were introduced in Duchon, Flajolet, Louchard, and Schaeffer \cite{MR2095975} and were developed further in Flajolet, Fusy and Pivoteau \cite{MR2498128}.
Following these sources we will briefly recall the theory of Boltzmann samplers to the extend required for the applications in this paper. Let $\cF \ne 0$ be an analytic species of structures and $F$ its exponential generating function. Given a parameter $x>0$ such that $0 < F(x) < \infty$, a Boltzmann sampler $\Gamma F(x)$ is a random generator that draws an object $\gamma \in \cF$ with probabilty 
\[
\Pr{\Gamma F(x) = \gamma} =  \frac{x^{|\gamma|}}{F(x) |\gamma|!}.
\]
In particular, if we condition on a fixed output size $n$, we get the uniform distribution on~$\cF_n$. We describe Boltzmann samplers using an informal pseudo-code notation. Given a specification of the species of structures $\cF$ in terms of other species using the operations of sums, products and composition, we obtain a Boltzmann sampler for $\cF$ in terms of samplers for the other species involved. The rules for the construction of Boltzmann samplers are summarized in Table~\ref{tb:boltzmann}. We let $\Bern{p}$ and $\Pois{\lambda}$ denote Bernoulli and Poisson distributed generators. 
{
\small
\renewcommand{\arraystretch}{1.15}
\begin{table}[h]
	\centering
	    \begin{tabular}{| l | l | l |}
	    	\hline
			$\cF = \cA + \cB$
					& \IF $\Bern{ A(x)/(A(x) + B(x)) }$ \THEN \RETURN $\Gamma A(x)$ \\
					& \ELSE \RETURN $\Gamma B(x)$
			\\ \hline 
			$\cF = \cA \cdot \cB$ 
					& \RETURN $(\Gamma A(x), \Gamma B(x))$ relabeled uniformly at random
			\\ \hline 
			$\cF = \cA \circ \cB$
					& $\gamma \leftarrow \Gamma A(y)$ with $y=B(x)$\\
					& \FOR $i = 1$ \TO $|\gamma|$ \\
					& \qquad $\gamma_i \leftarrow \Gamma B(x)$ \\
					& \RETURN $(\gamma, (\gamma_i)_i)$ relabeled uniformly at random
			\\ \hline 
			$\cF = \Set$
					& $m \leftarrow \Pois{x}$\\
					& \RETURN the unique structure of size $m$\\
			\hline	
   		\end{tabular}
	\caption{Rules for the construction of Boltzmann samplers.}
	\label{tb:boltzmann}
\end{table}
}

Note that if $\cF = \cA \mu \cB$ with  $\cA, \cB \ne 0$, $\mu \in \{+,\cdot, \circ\}$ and $0 < \cF(x) < \infty$, then the samplers of $\cA$ and $\cB$ are almost surely called with valid parameters, since the coefficients of all power-series involved are nonnegative.

Given a combinatorial specification $\cY \simeq \cH(\cX, \cY)$ satisfying the conditions of Theorem~\ref{te:implicitspecies}
we may apply the rules above to construct a recursive Boltzmann sampler that is guaranteed to terminate almost surely. This allows us to construct a Boltzmann sampler for block-stable graph classes. More specifically, let $\cC$ be a block-stable class of connected graphs such that the radius of convergence $\rhoc$ of the generating series $C(z)$ is positive. The rooted class $\cCb$ has a combinatorial specification given in~(\ref{eq:blockstability_classes}) in terms of the subclass $\cB$ of edges and 2-connected graphs. By Lemma~\ref{le:cfinite}, we know that $y=C^\bullet(\rhoc)$ and $\lambdac = B'(y)$ are finite. We obtain the following sampler which was used before in the study of certain block-stable graph classes, see for example \cite{MR2675698}. 
\begin{corollary}
\label{co:vanillaboltzmann}
Let $\cC$ be a block-stable class of connected graphs, $\cB \ne 0$ its subclass of all graphs that are 2-connected or a single edge. The following recursive procedure terminates almost surely and samples according to the Boltzmann distribution for $\cCb$ with parameter $\rhoc$.
\end{corollary}
\begin{tabular}{ll}
$\Gamma C^\bullet(\rhoc)$:  
		& $\gamma \leftarrow $ a single root vertex \\
		& $m \leftarrow \Pois{\lambda}$ \\
		& \FOR $k = 1 \ldots m$ \\
		& \qquad $B \leftarrow \Gamma B'(y)$, drop the labels \\
		& \qquad merge $\gamma$ with the $*$-vertex of $B$\\
		& \qquad \FOR \EACH non $*$-vertex $v$ of $B$ \\
		& \qquad \qquad $C \leftarrow \Gamma C^\bullet(\rhoc)$, drop the labels \\
		& \qquad \qquad merge $v$ with the root of $C$ \\
		& \RETURN $\gamma$ relabeled uniformly at random \\
\end{tabular}

\subsection{Subcritical Graph Classes}
\label{sec:subcri}
Let $\cC$ be a block-stable class of connected graphs and $\cB$ its subclass of all graphs that are 2-connected or a single edge with its ends. Assume that $\cB$ is nonempty and analytic, hence $\cC$ is analytic as well by Proposition~\ref{pro:analytic}. Denote by $\rhoc$ and $\rhob$ the radii of convergence of the corresponding exponential generating series $C(z)$ and $B(z)$. By Lemma~\ref{le:cfinite}, we know that $\rhoc$,  $y=C^\bullet(\rhoc)$ and $\lambdac = B'(y)$ are finite quantities.
The following proposition provides a coupling of a Boltzmann-distributed random graph drawn from the class $\cC$ with a Galton-Watson tree. This will play a central role in the proof of the main theorem.

\begin{proposition}
\label{pro:coupling}
Let $(\mT, \alpha)$ denote the enriched tree corresponding to the Boltzmann Sampler $\Gamma C^\bullet(\rhoc)$ given in Corollary \ref{co:vanillaboltzmann}. Then the rooted labeled unordered tree $\mT$ is distributed like the outcome of the following process:
\begin{enumerate}
\item[1.] Draw a Galton-Watson tree with offspring distribution $\xi$ given by the probability generating function $\varphi(z) = \exp(B'(yz) - \lambdac)$.
\item[2.] Distribute labels uniformly at random.
\item[3.] Discard the ordering on the offspring sets.
\end{enumerate}
\end{proposition}
\begin{proof}
This follows from the fact that every instance of the sampler $\Gamma C^\bullet(\rhoc)$ given in Corollary~\ref{co:vanillaboltzmann} calls itself recursively $|\Gamma(\Set \circ \cB')(y)|$ many times. 
\end{proof}
Let $\xi$ denote the offspring distribution given in Proposition~\ref{pro:coupling}.
As discussed above, the rules governing Boltzmann samplers guarantee that the sampler $\Gamma \cC^\bullet(\rhoc)$ terminates almost surely. Hence we have
$1 \ge \Ex{\xi} = \varphi'_\ell(1) = yB''(y) =  B'^\bullet(y)$
and in particular $y \le \rhob$. We define subcriticality depending on whether this inequality is strict.
\begin{definition}
A block-stable class of connected graphs $\cC$ is termed subcritical if $y < \rhob$.
\end{definition}
\noindent Prominent examples of subcritical graph classes are trees, outerplanar graphs and series-parallel graphs; the class of planar graphs does not fall into this framework \cite{MR2873207, MR2534261}, i.e.\ it satisfies $y = \rhob$. The following lemma was proved in Panagiotou and Steger \cite[Lem.\ 2.8]{MR2675698} by analytic methods.
\begin{lemma}
\label{le:clarifysubc}
If $B'^\bullet(\rhob) \ge 1$, then $B'^\bullet(y)=1$. If $B'^\bullet(\rhob) \le 1$, then $y=\rhob$. In particular, $\cC$ is subcritical if and only if $B'^\bullet(\rhob) > 1$.  
\end{lemma}
\noindent Thus, if $B'^\bullet(\rhob) \ge 1$, then the offspring distribution $\xi$ has expected value $1$ and variance
\[
\sigma^2 = 1 + B'''(y) y^2 = \Ex{|\Gamma B'^\bullet(y)|}
\]
with $\Gamma B'^\bullet(y)$ denoting a Boltzmann sampler for the class $\cB'^\bullet$ with parameter $y$. By Proposition~\ref{pro:coupling} the size of the outcome of the sampler $\Gamma \cC^\bullet(\rhoc)$ is distributed like the size of a $\xi$-Galton-Watson tree. Hence, we may apply Lemma~\ref{le:cycap} to obtain the following result, which was shown in~\cite{MR2873207} under stronger assumptions.
\begin{corollary}
\label{co:sizepoly}
Let $\cC$ be an analytic block-stable class of graphs, and let $\xi$ be the distribution from Proposition~\ref{pro:coupling}. Suppose that $B'^\bullet(\rhob) \ge 1$ and $B'''(y) < \infty$, i.e.\ $\xi$ has finite variance. Let $d = \spa(\xi)$. Then, as $n \equiv 1 \mod d$ tends to infinity,
\[
\Pr{|\Gamma \cC^\bullet(\rhoc)| = n} \sim{} \frac{d}{\sqrt{2 \pi \Ex{|\Gamma B'^\bullet(y)|}}} n^{-3/2}
\quad\text{ and }\quad
|\cC_n| \sim \frac{yd}{\sqrt{2 \pi \Ex{|\Gamma B'^\bullet(y)|}}} n^{-5/2} \rhoc^{-n} n!.
\]
\end{corollary}

\subsection{Deviation Inequalities}
\label{sec:deviation}
We will make use of the following moderate deviation inequality for one-dimensional random walk found in most textbooks on the subject.

\begin{lemma}
\label{le:deviation}
Let $(X_i)_{i \in \ndN}$ be family of independent copies of a real-valued random variable $X$ with  $\Ex{X} = 0$. Let $S_n = X_1 + \ldots + X_n$. Suppose that there is a $\delta >0$ such that $\Ex{e^{\theta X}} < \infty$ for $|\theta| < \delta$. Then there is a $c>0$ such that for every $1/2 < p < 1$ there is a number $N$ such that for all $n \ge N$ and $0 < \epsilon < 1$ 
\[
\Pr{|S_n/n^p| \ge \epsilon } \le 2 \exp(- c \epsilon^2 n^{2p-1}).
\] 
\end{lemma}

\section{A Size-Biased Random $\cR$-enriched Tree}
\label{sec:pointedenr}
Let $\cC$ be an analytic block-stable class of connected graphs and $\cB \ne 0$ its subclass of graphs that are 2-connected or a single edge. As before we let $\rhoc$ denote the radius of convergence of the exponential generating series $C(z)$ and set $y = C^\bullet(\rhoc)$. Recall that by Corollary~\ref{co:encor} the class $\cCb$ may be identified with the class of $\cR$-enriched trees with $\cR := \Set \circ \cB'$, i.e.\ pairs $(T, \alpha)$ with $T \in \cTb$ a rooted labeled unordered tree and $\alpha$ a function that assigns to each $v \in V(T)$ a (possibly empty) set $\alpha(v)$ of derived blocks whose vertex sets partition the offspring set of the vertex $v$.

An important ingredient in our forthcoming argumets will be an accurate description of the distribution of the blocks on sufficiently long paths in random graphs from $\cC$. In order to study this distribution we will make use of a special case of a \emph{size-biased} random $\cR$-enriched tree. This construction is based on the size-biased Galton Watson tree introduced in \cite{MR3077536} and has several other applications and implications. A thorough study in a more general setting will be presented in Stufler~\cite{RandEnr}.

Recall that $\cA_\cR$ has the decomposition $\cA_\cR \simeq \cX \cdot \cR(\cA_\cR)$. By the rules~\eqref{eq:distrLaw} and~\eqref{eq:chainRule} governing operations on species we obtain algebraically
\begin{align*}
\cA_\cR^\bullet &\simeq \cA_\cR + \cX \cdot \cR'(\cA_\cR)\cdot \cA_\cR^\bullet \\
&\simeq \cA_\cR + \cX \cdot \cR'(\cA_\cR)\cdot \cA_\cR + (\cX \cdot \cR'(\cA_\cR))^2 \cdot \cA_\cR^\bullet \\
& \ldots \\
& \simeq  \sum_{\ell \ge 0} (\cX \cdot \cR'(\cA_\cR))^\ell \cA_\cR.
\end{align*}
The above calculation corresponds to taking a direct limit in the category of species, either directly or alternatively by application of Joyal's Implicit Species Theorem. Here $\cA_\cR^{(\ell)} := (\cX \cdot \cR'(\cA_\cR))^\ell \cA_\cR$ corresponds to the subspecies of all enriched trees $(T, \alpha)$ with a distinguished vertex $r$ such that $r$ has height $\ell$ in $T$.
It follows from the definition of the Boltzmann distribution that for any pointed enriched tree $(A,r) \in \cA_\cR^{(\ell)}$ we have that
\begin{align}
\label{eq:pureawesomeness}
\Pr{\Gamma A_R^{(\ell)}(\rho) = (A,r)} = \left(\rho R'(y)\right)^{-\ell}\Pr{\Gamma A_R(
\rho) = A}.
\end{align}
Translating the combinatorial specification for $\cA_\cR^{(\ell)}$ into a Boltzmann sampler $\Gamma A_R^{(\ell)}(\rho)$ yields the following procedure which we call the {\em size-biased $\cR$-enriched tree} (see also Figure~\ref{fi:pntenriched}).
Any vertex is either \emph{normal} or \emph{mutant}, and we start with a single mutant root. Normal vertices have an independent copy of $\Gamma R(y)$ as offspring. Mutant nodes have an independent copy of $\Gamma R^\bullet(y)$ as offspring and the root in the $\cR^\bullet$ object is declared mutant, unless its the $\ell$th copy of $\Gamma R^\bullet(y)$. By the theory of recursive Boltzmann samplers obtained from combinatorial specifications this procedure terminates almost surely. The sampler $\Gamma A_R^{(\ell)}(\rho)$ is obtained by additionally distributing labels uniformly at random.

\begin{figure}[ht]
	\centering
	\begin{minipage}{1.0\textwidth}
  		\centering
  		\includegraphics[width=0.6\textwidth]{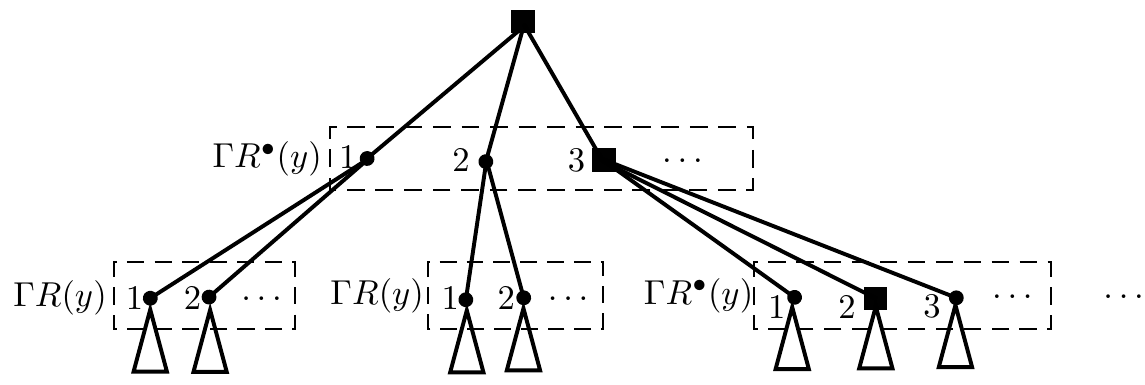}
  		\caption{Illustration of the sampler for the size-biased $\cR$-enriched tree.}
  		\label{fi:pntenriched}
	\end{minipage}
\end{figure}

We call the path connecting the inner root with the outer root in an $\cA_\cR^\bullet$-object the {\em spine}. Note that the $\cR$-objects along the spine of the random enriched tree $\Gamma A_R^{(\ell)}(\rho)$ are drawn according to $\ell$ independent copies of $\Gamma R^\bullet(y)$.

In our setting we have that $\cR = \Set \circ \cB'$, where $\cB \neq 0$ denotes the subclass of blocks of the block-stable class $\cC$. Using~\eqref{eq:chainRule} we obtain
\[
\cR^\bullet \simeq (\Set \circ \cB') \cdot \cB'^\bullet
\]
and the sampler $\Gamma R^\bullet(y)$ is given by independent calls of $\Gamma (\Set \circ \cB')(y)$ and $\Gamma \cB'^\bullet(y)$. Hence the blocks along the spine are drawn according to $\ell$ independent copies of $\Gamma \cB'^\bullet(y)$.

Equation~(\ref{eq:pureawesomeness}) allows us to relate properties of $\Gamma A_R^{(\ell)}(\rho)$ to properties of a uniformly random chosen enriched tree of a given size. We are going to apply the following general lemma in Section~\ref{sec:convergence} in order to show that the blocks along sufficiently long paths in random graphs behave asymptotically like the spine of $\Gamma A_R^{(\ell)}(\rho)$ for a corresponding $\ell$.
\begin{lemma}
\label{le:bipenr}
Let $\cE$ be a property of pointed $\cR$-enriched trees (i.e.\ a subset of $\cA_\cR^\bullet$) and let $n \in \ndN$ be such that $\cA_\cR[n]$ is nonempty. Consider the function
\[
f: \cA_\cR[n] \to \ndR, \, \, A \mapsto \sum_{v \in [n]} \one_{(A,v) \in \cE}
\]
counting the number of ``admissible'' outer roots with respect to $\cE$. Let $\mA_n \in \cA_\cR[n]$ be drawn uniformly at random. Then
\[
\Ex{f(\mA_n)}= \Pr{|\Gamma A_R(\rho)| = n}^{-1} \sum_{\ell=0}^{n-1} \left(\rho R'(y)\right)^{\ell} \Pr{\Gamma A_R^{(\ell)}(\rho) \text{ has size $n$ and satisfies $\cE$}}.
\]
\end{lemma}
\begin{proof}
First, observe that 
\[
\sum_{v=1}^n \Pr{ (\mA_n,v) \in \cE} = \sum_{\ell=0}^{n-1} \sum_{(A,r) \in \cE \cap \cA_\cR^{(\ell)}[n]} \Pr{\mA_n = A}.
\]
By (\ref{eq:pureawesomeness}) we have for all $(A,r) \in \cE \cap \cA_\cR^{(\ell)}[n]$ that
\[
\Pr{\Gamma A_R(\rho) = A \mid |\Gamma A_R(\rho)| = n} = \left(\rho R'(y)\right)^{\ell}  \Pr{\Gamma A_R^{(\ell)}(\rho) = (A,r)} \Pr{|\Gamma A_R(\rho)| = n}^{-1}.
\]
This proves the claim.
\end{proof}

\section{Convergence Towards the CRT}
\label{sec:convergence}
Let $\cC$ be an analytic block-stable class of connected graphs and $\cB \ne 0$ its subclass of all graphs that are 2-connected or a single edge. We let $\rhoc>0$ denote the radius of convergence of the exponential generating series $C(z)$ and set $y = C^\bullet(\rhoc)$. As before we identify $\cC^\bullet$ with the class $\cA_\cR$ of $\cR$-enriched trees with $\cR = \Set \circ \cB'$.
By Proposition~\ref{pro:coupling} we know that if we draw an $\cR$-enriched tree $(\mT, \alpha)$ according to the Boltzmann distribution with parameter $\rhoc$, then $\mT$ is distributed like a $\xi$-Galton-Watson tree with $\xi := |\Gamma(\Set \circ \cB')(y)|$, relabeling uniformly at random and discarding the ordering on the offspring sets.

Throughout this section let $n \equiv 1 \mod \spa(\xi)$ denote a large enough integer such that the probability of a $\xi$-GWT having size $n$ is positive. Let $\mC_n \in \cC_n$ be drawn uniformly at random and generate $\mC_n^\bullet \in \cC^\bullet_n$ by uniformly choosing a root from $[n]$. We let $(\mT_n, \alpha_n)$ be the corresponding enriched tree.

For any pointed derived block $B \in \cB'^\bullet$ we let $\shp(B) := d_B(*, \text{root})$ denote the length of a shortest path connecting the $*$-vertex with the root.
In this section we prove our main result:
\begin{theorem}
\label{te:maintheorem}
Let $\cC$ be a subcritical class of connected graphs. Then 
\[
\frac{\sigma}{2 \sca \sqrt{n}} \mC_n^\bullet \convdislong \CRT \quad \text{and}\quad \frac{\sigma}{2 \sca \sqrt{n}} \mC_n \convdislong \CRT
\]
with respect to the (pointed) Gromov-Hausdorff metric. The constants are given by  $\sigma^2 = \Ex{|\mB|}$ and $\sca = \Ex{|\shp(\mB)|}$ with $\mB \in \cB'^\bullet$ a random block drawn according to the Boltzmann distribution with parameter $y=\cC^\bullet(\rho)$, and in particular $\sigma^2 = 1 + B'''(y) y^2$.
\end{theorem}
As a consequence we obtain the limit distributions for the height and diameter of $\mC_n^\bullet$.
\begin{corollary}
\label{cor:hddistr}
Let $\cC$ be a subcritical class of connected graphs. Then the rescaled height $\frac{\sigma}{2 \sca \sqrt{n}} \He(\mC_n^\bullet)$ and diameter $\frac{\sigma}{2 \sca \sqrt{n}} \Di(\mC_n)$ converge in distribution to $\He(\CRT)$ and $\Di(\CRT)$, i.e.\ for all $x>0$, as $n$ tends to infinity
\begin{align*}
\Pr{\He(\mC_n^\bullet) > \frac{2 \sca \sqrt{n}}{\sigma}x} &\to 2 \sum_{k=1}^\infty (4k^2x^2-1)\exp(-2k^2x^2), \\
\Pr{\Di(\mC_n) > \frac{2 \sca \sqrt{n}}{\sigma}x} &\to \sum_{k=1}^\infty (k^2-1)\left(\frac{2}{3}k^4x^4 -4k^2x^2 +2\right)\exp(-k^2x^2/2).
\end{align*}
Moreover, all moments converge as well. In particular
\begin{align*}
\Ex{\Di(\mC_n)} \sim \frac{2^{5/2} \sca}{3\sigma} \sqrt{\pi n} \sim \frac{4}{3} \Ex{\He(\mC_n^\bullet)}.
\end{align*}
Expressions for arbitrarily high moments are given in (\ref{eq:exheight}) and (\ref{eq:exdiam}).
\end{corollary}
\begin{proof}
The limiting distributions are given in (\ref{eq:disheight}) and (\ref{eq:disdiam}). In order to show convergence of the moments we argue that the rescaled height and diameter are bounded in the space $L^{p}$ for all $1 < p < \infty$. This follows for example from the subgaussian tail-bounds of Theorem~\ref{te:tailbound} given in Section \ref{sec:tailbounds} below (note that the proof of Theorem~\ref{te:tailbound} does not depend on the results in this section).
\end{proof}
In the following we are going to prove Theorem~\ref{te:maintheorem}. The idea is to show that the pointed Gromov-Hausdorff distance of $\mC_n^\bullet$ and $\sca \mT_n$ is small with high probability and use the convergence of $\mT_n$ towards a multiple of the CRT $\CRT$.
\begin{definition}
Let $C\in\cC$. For any $x,y \in V(C)$ set $\bar{d}_C(x,y) := d_T(x,y)$ with with $(T,\alpha)$ the enriched tree corresponding to $(C,x)$, i.e.\ $C$ rooted at the vertex $x$.
\end{definition}
Less formally speaking, $\bar{d}_C(x,y)$ denotes the minimum number of blocks required to cover the edges of a shortest path linking $x$ and $y$. It takes a moment to see that if $(T,\alpha)$ corresponds to the rooted graph $(C,z)$, then $\bar{d}_C(x,y) \le d_T(x,y) \le \bar{d}_C(x,y) + 1$ for all $x,y \in V(C)$. In particular, $\bar{d}_C$ is a metric: the triangle inequality holds since $d_T$ is a metric and
\[
\bar{d}_C(x,y) \le d_T(x,y) \le d_T(x,z) + d_T(y,z) = \bar{d}_C(x,z) + \bar{d}_C(y,z).
\] 
In the following lemma we apply the results on pointed enriched trees of Section~\ref{sec:pointedenr}.
\begin{lemma}
\label{le:expansion}
Let $\cC$ be a subcritical class of connected graphs and set $\sca = \Ex{\shp(\Gamma B'^\bullet(y))}$. Then for all $s>1$ and $0 < \epsilon < 1/2$ with $2\epsilon s>1$
we have with high probability that all $x,y \in V(\mC_n)$ with $\bar{d}_{\mC_n}(x,y) \ge \log^s(n)$ satisfy $|d_{\mC_n}(x,y) - \sca \bar{d}_{\mC_n}(x,y)| \le \bar{d}_{\mC_n}(x,y)^{1/2 + \epsilon}$.
\end{lemma}
\begin{proof}
We denote $L_n = \log^s(n)$ and $t_\ell = \ell^{1/2 + \epsilon}$. Let $\cE \subset \cA_{\cR}^\bullet \simeq \cC^{\bullet \bullet}$ with $\cR = \Set \circ \cB'$ denote the set of all bipointed graphs or pointed enriched trees $((C,x),y) \simeq ((T, \alpha), y)$, where we call $x$ the {\em inner root} and $y$ the {\em outer root}, such that 
\[
d_T(x,y) \ge L_{|T|} \quad \text{and} \quad |d_C(x,y) - \sca d_T(x,y)| > t_{d_T(x,y)}.
\]
We will bound the probability that there exist vertices $x$ and $y$ with $((\mC_n,x),y) \in \cE$. First observe that
\[
\sum_{x,y \in [n]} \Pr{((\mC_n, x),y) \in \cE} = \sum_{((C,x),y) \in \cE} \Pr{ \mC_n = C} =  n \sum_{y=1}^n \Pr{(\mC_n^\bullet,y) \in \cE}.
\]
By assumption we may apply Corollary~\ref{co:sizepoly} to obtain $\Pr{|\Gamma \cC^\bullet(\rho)| = n} = \Theta(n^{-3/2})$. Moreover, Lemma~\ref{le:clarifysubc} asserts that $B'^\bullet(y)=1$ and thus, with Lemma~\ref{le:cfinite}
\[
	\rhoc R'(y)
	= \rhoc B''(y) e^{B'(y)}
	= y B''(y)
	=1.
\]
Hence, by applying Lemma~\ref{le:bipenr} we obtain that
\begin{align*}
\label{eq:proofmainlm}
\Pr{((\mC_n,x),y) \in \cE \text{ for some $x, y$}} \le O(n^{5/2}) \sum_{\ell=L_n}^{n-1} \Pr{\Gamma A_R^{(\ell)}(\rhoc) \text{ has size $n$ and satisfies $\cE$}}.
\end{align*}
The height of the outer root in the bipointed graph corresponding to $\Gamma A_R^{(\ell)}(\rhoc)$ is distributed like the sum of $\ell$ independent random variables, each distributed like the distance of the $*$-vertex and the root in the corresponding derived block of $\Gamma (\Set \circ \cB')^\bullet(y)$. Since $(\Set \circ \cB')^\bullet \simeq (\Set \circ \cB')\cdot \cB'^\bullet$, these variables are actually $\shp(\Gamma B'^\bullet(y))$-distributed. Hence \begin{align*}
\Pr{\Gamma A_R^{(\ell)}(\rhoc) \in \cE, |\Gamma A_R^{(\ell)}(\rhoc)|=n} &\le \Pr{|\eta_1 + \ldots + \eta_\ell - \ell \Ex{\eta_1}| > t_\ell}
\end{align*} with $(\eta_i)_i$ i.i.d.\ copies of $\eta := \shp(\Gamma B'^\bullet(y))$. Clearly we have that $\eta \le |\Gamma B'^\bullet(y)|$. Since  $\cC$ is subcritical it follows that there is a constant $\delta>0$ such that $\Ex{e^{\eta \theta}}< \infty$ for all $\theta$ with $|\theta| \le \delta$. Hence we may apply the standard moderate deviation inequality for one-dimensional random walk stated in Lemma \ref{le:deviation} to obtain for some constant $c>0$
\[
\Pr{((\mC_n,x),y) \in \cE \text{ for some $x, y$}} \le O(n^{7/2}) \exp( - c (\log n)^{2s\epsilon})= o(1).
\]
\end{proof}
It remains to clarify what happens if $\bar{d}_{\mC_n}$ is small. We prove the following statement for random graphs from block-stable classes that are not necessarily subcritical. 
\begin{proposition}
\label{pro:shortpaths}
Let $\cC$ be a block-stable class of connected graphs. Suppose that $B'^\bullet(y)=1$ and the offspring distribution $\xi$ has finite second moment, i.e.\ $B'''(y) < \infty$. Let $\lhb(\mC_n)$ denote the size of the largest block in $\mC_n$,
\begin{enumerate}
	\item For any $x,y \in \mC_n$ we have $d_{\mC_n}(x,y) \le \bar{d}_{\mC_n}(x,y) \lhb(\mC_n)$.
	\item If the offspring distribution $\xi$ is bounded, then so is $\lhb(\mC_n)$. Otherwise, for any sequence $K_n$ we have $\Pr{\lhb(\mC_n) \ge K_n} = O(n) \Pr{\xi \ge K_n}$.
\end{enumerate}
\end{proposition}
\begin{proof}
We have that $d_{\mC_n} \le \bar{d}_{\mC_n}(\lhb(\mC_n)-1)$ and $\lhb(\mC_n) = \lhb(\mC_n^\bullet) \le \Delta(\mT_n)+1$ with $\Delta(\mT_n)$ denoting the largest outdegree.
Recall that $\Delta(\mT_n)$ is distributed like the maximum degree of a $\xi$-Galton-Watson tree conditioned to have $n$ vertices.  By assumption, the offspring distribution $\xi$ has expected value $\Ex{\xi}=B'^\bullet(y)=1$ and finite variance.

If $\xi$ is bounded then so is the largest outdegree of $\mT_n$. Otherwise, as argued in the proof of \cite[Eq.\ (19.20)]{MR2908619}, for any sequence $K_n$ 
\begin{align}
\label{eq:maxdegree}
\Pr{\Delta(\mT_n) \ge K_n} \le (1 + o(1))n\Pr{\xi \ge K_n}.
\end{align}
Applying (\ref{eq:maxdegree}) yields $\Pr{\Delta(\mT_n) \ge K_n} \le (1+o(1)) n \Pr{\xi \ge K_n}$ for any sequence $K_n$.
\end{proof}
Note that if $\cC$ is subcritical then this implies that $\lhb(\mC_n) = O(\log n)$ with high probability: the definition of the Boltzmann model and the fact that $y$ is smaller than the radius of convergence of $B(z)$ guarantee that there is a constant $\beta < 1$ such that
\[
	\Pr{\xi = k}
	= \Pr{|\Gamma(\Set \circ \cB')(y)| = k}
	= O(\beta^k).
\]
Combined with the bounds of Lemma \ref{le:expansion} this yields the following concentration result.
\begin{corollary}
\label{cor:short}
Let $\cC$ be a subcritical class of connected graphs. Then for all $s>1$ and $0 < \epsilon < 1/2$ with $2\epsilon s>1$ we have with high probability that for all vertices $x, y \in V(\mC_n)$ 
\[|d_{\mC_n}(x,y) - \sca \bar{d}_{\mC_n}(x,y)| \le \bar{d}_{\mC_n}(x,y)^{1/2 + \epsilon} + O(\log^{s+1}(n)).\]
\end{corollary}

We may now prove the main theorem.

\begin{proof}[Proof of Theorem~\ref{te:maintheorem}]
Recall that $\bar{d}_{\mC_n} \le d_{\mT_n} \le \bar{d}_{\mC_n}+1$. By Corollary~\ref{cor:short}, Proposition~\ref{pro:distortion}, and considering the distortion of the identity map as correspondence between the vertices of $\mT_n$ and $\mC_n^\bullet$, it follows that with high probability \[d_{\text{GH}}(\mC_n^\bullet / (\sca \sqrt{n}), \mT_n / \sqrt{n}) \le \Di(\mT_n)^{3/4} / \sqrt{n} + o(1).\] 
Using the tail bounds (\ref{eq:gwttail}) for the diameter $\Di(\mT_n)$ we obtain that
$
d_{\text{GH}}(\mC_n^\bullet / (\sca \sqrt{n}), \mT_n / \sqrt{n})
$ converges in probability to zero. Recall that the variance of the offspring distribution $\xi$ is given by $\sigma^2 = \Ex{|\Gamma B'^\bullet(y)|}$. By Theorem~\ref{te:gwtconv} we have that
$
\frac{\sigma}{2 \sqrt{n}}\mT_n \convdislong \CRT
$
and thus
$
\frac{\sigma}{2 \sca \sqrt{n}}\mC_n^\bullet \convdislong \CRT.
$
\end{proof}

\section{Subgaussian Tail Bounds for the Height and Diameter}
\label{sec:tailbounds}
In this section we prove subgaussian tail bounds for the height and diameter of the random graphs $\mC_n^\bullet$ and $\mC_n$. Our proof builds on results obtained in~\cite{MR3077536}. Recall that $(\mT_n, \alpha_n)$ denotes the enriched tree corresponding to the graph $\mC_n^\bullet$ and that $\mT_n$ has a natural coupling with a $\xi$-Galton-Watson conditioned on having size $n$, see Proposition~\ref{pro:coupling}. With (slight) abuse of notation we also write $\mT_n$ for the conditioned $\xi$-Galton-Watson tree within this section. We prove the following statement for random graphs from block-stable classes that are not necessarily subcritical. 
\begin{theorem}
\label{te:tailbound}
Let $\cC$ be a block-stable class of connected graphs. Suppose that  $\cC$ satisfies $B'^\bullet(y)=1$ and the offspring distribution $\xi$ has finite variance, i.e.\ $B'''(y)<1$. Then there are $C,c>0$ such that for all $n,x \ge 0$
\[
\Pr{\Di(\mC_n) \ge x} \le C \exp(-cx^2/n) \quad \text{and} \quad \Pr{\He(\mC_n^\bullet) \ge x} \le C \exp(-cx^2/n).
\]
\end{theorem}
As $\He(\mT_n) \le \He(\mC_n^\bullet)$, Inequality~\eqref{eq:gwttail} also yields a lower tail bound for the height of $\mC_n^\bullet$.
\begin{corollary}
There are constants $C,c>0$ such that for all $x\ge0$ and $n$
\[
\Pr{\He(\mC_n^\bullet) \le x} \le C \exp(-c(n-2)/x^2).
\]
\end{corollary}
As a main ingredient in our proof we consider the \emph{lexicographic depth-first-search (DFS)} of the plane tree $\mT_n$ by labeling the vertices in the usual way (as a subtree of the Ulam-Harris tree) by finite sequences of integers and listing them in lexicographic order $v_0, v_1, \ldots, v_{n-1}$. The search keeps a queue of $Q_i^d$ nodes with $Q_0^d=1$ and the recursion $$Q_i^d = Q_{i-1}^d -1 + d_{\mT_n^p}^+(v_{i-1}).$$
The mirror-image of $\mT_n$ is obtain by reversing the ordering on each offspring set and the \emph{reverse} DFS $Q_i^r$ is defined as the DFS of the mirror-image. Then $(Q_i^d)_{0 \le i \le n}$ and $(Q_i^r)_{0 \le i \le n}$ are identically distributed and satisfy the following bound given in \cite[Ineq.\ (4.4)]{MR3077536}:
\begin{align}
\label{eq:dfs}
\Pr{\max_{j} Q_j \ge x} \le C \exp(-c x^2 /n)
\qquad \text{where } Q_j \in \{Q_j^d, Q_j^r\}
\end{align}
with $C, c>0$ denoting some constants that do not depend on $x$ or $n$.

\begin{proof}[Proof of Theorem~\ref{te:tailbound}]
Since $\Di(\mC_n) \le 2 \He(\mC_n^\bullet)$ it suffices to show the bound for the height. Let $h \ge 0$. If $\He(\mC_n^\bullet) \ge h$ then there exists a vertex whose height equals $h$. Consequently, we may estimate $\Pr{\He(\mC_n^\bullet) \ge h} \le \Pr{\cE_1} + \Pr{\cE_2}$ with $\cE_1$ (resp.\ $\cE_2$) denoting the event that there is a vertex $v$ such that $h_{\mC_n^\bullet}(v) = h$ and $h_{\mT_n}(v) \ge h/2$ (resp.\ $h_{\mT_n}(v) \le h/2$).
By the tail bound \eqref{eq:gwttail} for the height of Galton-Watson trees we obtain
\[
\Pr{\cE_1} \le \Pr{\He(\mT_n) \ge h} \le C_2 \exp(-c_2 h^2/(4n))
\]
for some constants $C_2, c_2 > 0$. In order to bound $\Pr{\cE_2}$ suppose that there is a vertex $v$ with height $h_{\mC_n^\bullet}(v) = h$ and $h_{\mT_n}(v) \le h/2$. If $a$ is a vertex of $\mT_n$ and $b$ one of its offspring, then $d_{\mC_n^\bullet}(a,b) \le d_{\mT_n}^+(a)$. Hence 
\[
\sum_{u \succ v} d_{\mT_n}^+(u) \ge h_{\mC_n^\bullet}(v) = h
\]
with the sum index $u$ ranging over all ancestors of $v$.
Consider the lexicographic depth-first-search $(Q_i^d)_i$ and reverse depth-first-search $(Q_i^r)_i$ of $\mT_n^p$. Let $j$ (resp.\ $k$) denote the index corresponding to the vertex $v$ in the lexicographic (resp.\ reverse lexicographic) order. It follows from the definition of the queues that if $\cE_2$ occurs
\[
Q_j^d + Q_k^r = 2 + \sum_{u \succ v} d_{\mT_n^p}^+(u) - h_{\mT_n}(v) \ge h/2
\]
and hence $\max(Q_j^d, Q_k^r) \ge h/4$.  Since $Q_j^d$ and $Q_k^r$ are identically distributed it follows by~\eqref{eq:dfs} that
\begin{align*}
\Pr{\cE_2} &\le \Pr{\max_i(Q_i^d) \ge h/4} + \Pr{\max_i(Q_i^r) \ge h/4} \\
& \le 2 \Pr{\max_i(Q_i^d) \ge h/4} \\
& \le 2 C \exp(-c h^2/(16n)).
\end{align*}
This concludes the proof.
\end{proof}

\section{Extensions}
\label{sec:extensions}
In the following we use the notation from Section \ref{sec:convergence}.
\subsection{First Passage Percolation}
\label{sec:weights}
Let $\omega>0$ be a given random variable such that there is a $\delta > 0$ with $\Ex{e^{\theta \omega}} < \infty$ for all $\theta$ with $|\theta| \le \delta$. For any graph $G$ we may consider the random graph $\hat{G}$ obtained by assigning to each edge $e \in E(G)$ a weight $\omega_e$ that is an independent copy of $\omega$. The $d_{\hat{G}}$-distance  of two vertices $a$ and $b$ is then given by 
\[
d_{\hat{G}}(a,b) = \inf\Big\{ \sum_{e \in E(P)} \omega_e \mid \text{$P$ a path connecting $a$ and $b$ in $G$} \Big\}.
\]
Let $\hat{\sca} := \Ex{\shp(\hat{\mB})}$ with $\mB$ drawn according to the Boltzmann sampler $\Gamma B'^\bullet(y)$ and $\shp(\hat{\mB})$ denoting the $d_{\hat{\mB}}$-distance from the $*$-vertex to the root vertex. 
\begin{theorem}
Let $\cC$ be a subcritical class of connected graphs. We have that 
\[
\frac{\sigma}{2 \hat{\sca} \sqrt{n}} \hat{\mC}_n^\bullet \convdislong \CRT \quad \text{and}\quad \frac{\sigma}{2 \hat{\sca} \sqrt{n}} \hat{\mC}_n \convdislong \CRT
\]
with respect to the (pointed) Gromov-Hausdorff metric.
\end{theorem} 
\begin{proof}
For any $n$ let $K_n$ denote the complete graph with $n$ vertices. The idea is to generate $\hat{\mC}_n$ by drawing $\mC_n$ and $\hat{K}_n$ independently and assign the weights via the inclusion $E(\mC_n) \subset E(K_n)$. By considering subsets $\cE \subset \cC^{\bullet \bullet} \times \ndR^{\cup_nE(K_n)}$ we may easily prove a weighted version of Lemma~\ref{le:bipenr}, i.e. the probability that the random pair $(\mC_n^{\bullet \bullet}, \hat{K}_n)$ has some property $\cE$ is bounded by
$$O(n^{5/2}) \sum_{\ell=0}^{n-1} \Pr{|\Gamma C^{\bullet (\ell)}(\rho)|=n, (\Gamma C^{\bullet (\ell)}(\rho), \hat{K}_n) \in \cE}.$$ 
This implies that the blocks along sufficiently long paths in the random graphs $\hat{\mC}_n$ behave like independent copies of the weighted block $\hat{\mB}$ with $\mB$ drawn according to the Boltzmann sampler $\Gamma B'^\bullet(y)$. Hence, weighted versions of Lemma~\ref{le:expansion} and Proposition~\ref{pro:shortpaths} may be deduced analogously to their original proofs with $\hat{\kappa}$ replacing $\kappa$ and only minor modifications otherwise. Thus the scaling limit follows in the same fashion. 
\end{proof}

\subsection{Random Graphs given by Their Connected Components}
\label{sec:disconnected}

We study the case of an arbitrary graph consisting of a set of connected components. Let $\cG \simeq \Set \circ \cC$ denote a subcritical graph class given by its subclass $\cC$ of connected graphs. For simplicity we are going to assume that all trees belong to the class $\cC$. 

Consider the uniform random graph $\mG_n \in \cG_n$. Of course we cannot expect $\mG_n$ to converge to the Continuum Random Tree since may be disconnected with a probability that is bounded away from zero. Instead we study a uniformly chosen component $\mH_n$ of maximal size. We are going to show that $\mH_n / \sqrt{n}$ converges to a multiple of the CRT.
\begin{theorem}
\label{te:compconv}
Suppose  $\cC$ is subcritical class of connected graphs containing all trees. Then 
$
\frac{\sigma}{2 \sca \sqrt{n}} \mH_n \convdislong \CRT
$
with respect to the Gromov-Hausdorff metric, where $\sigma, \kappa$ are as in Theorem~\ref{te:maintheorem}.
\end{theorem}
We are going to use the known fact that with high probability the random graph $\mG_n$ has a unique giant component with size $n + O_p(1)$. This follows for example from \cite[Thm. 6.4]{MR2249274}.

\begin{lemma}
\label{le:sizecomp}
If $\cC$ contains all trees, then the size of a largest component satisfies $|\mH_n| = n + O_p(1)$.
\end{lemma}

\begin{proof}[Proof of Theorem~\ref{te:compconv}]
Let $f: \ndK \to \ndR$ be a bounded Lipschitz-continuous function defined on the space of isometry classes of compact metric spaces. We will show that
$
\Ex{f (\frac{\sigma}{2 \sca \sqrt{n}} \mH_n)} \to \Ex{f(\CRT)}
$
as $n$ tends to infinity.
Set $\Omega_n := \log n$. By Lemma~\ref{le:sizecomp} we know that $|\mH_n| = n + O_p(1)$. Hence with high probability we have that $n - |H_n| \le \Omega_n$ and thus
\[
\mathbb{E}\left[f\left(\frac{\sigma}{2 \sca \sqrt{n}} \mH_n\right)\right] = o(1) + \sum_{0 \le k \le \Omega_n}
\mathbb{E}\left[f\left(\frac{\sigma}{2 \sca \sqrt{n}} \mH_n\right) \mid |\mH_n| = n-k\right] \, \Pr{|\mH_n| = n-k}.
\]
The conditional distribution of $\mG_n$ given the sizes $(s_i)_i$ of its components is given by choosing components $K_i \in \cC[s_i]$ independently uniformly at random and distributing labels uniformly at random. In particular, as a metric space, $\mH_n$ conditioned on $|\mH_n| = n-k$ is distributed like the uniform random graph $\mC_{n-k}$. Thus, given $\epsilon >0$ we have for $n$ sufficiently large by Lipschitz-continuity 
\[
\mathbb{E}\left[f\left(\frac{\sigma}{2 \sca \sqrt{n}} \mH_n\right) ~\big|~ |\mH_n| = n-k \right] = \mathbb{E}\left[f\left(\frac{\sigma}{2 \sca \sqrt{n}} \mC_{n-k}\right)\right] \in \Ex{f(\CRT)} \pm \epsilon
\]
for all $0 \le k \le \Omega_n$.
Thus $| \Ex{f(\frac{\sigma}{2 \sca \sqrt{n}} \mH_n)} - \Ex{f(\CRT)}| \le \epsilon$ for sufficiently large $n$. Since $\epsilon>0$ was chosen arbitrarily it follows that $\Ex{f(\frac{\sigma}{2 \sca \sqrt{n}} \mH_n)} \to \Ex{f(\CRT)}$ as $n$ tends to infinity.
\end{proof}

\section{The Scaling Factor of Specific Classes}
\label{sec:examples}
In this section we apply our main results to several specific examples of subcritical graph classes. 
The notation that will be fixed throughout this section is as follows: $\cC$ denotes a subcritical class of connected graphs and $\cB$ its subclass of 2-connected graphs and edges. The radius of convergence of $C(z)$ is denoted by $\rhoc$. The constant $y = \cC^\bullet(\rhoc)$ is the unique positive solution of the equation
\[
yB''(y) = 1.
\]
By Lemma~\ref{le:cfinite} this determines $\rhoc = y \exp(-B'(y))$. Moreover, we set
\[
\sca = \Ex{\shp(\Gamma B'^\bullet(y))},
\]
i.e.\ the expected distance from the $*$-vertex to the root in a random block chosen according to the Boltzmann distribution with parameter $y$. We call $\sca$ the \emph{scaling factor} for $\cC$. The offspring distribution $\xi$ of the random tree corresponding to the sampler $\Gamma \cC^\bullet(y)$ has probability generating function $\varphi(z) = \exp(B'(yz)-\lambdac)$ with $\lambdac = B'(y)$, see Proposition~\ref{pro:coupling}. Its variance is given by
\[
\sigma^2 = 1 + B'''(y) y^2 = \Ex{|\Gamma B'^\bullet(y)|}.
\]
We let $d$ denote the span of the offspring distribution. By applying Corollary~\ref{cor:hddistr} we obtain 
\[
\Ex{\mH(\mC_n^\bullet)} / \sqrt{n} \to \sca \sqrt{2 \pi / \sigma^2} =: H \quad \text{as $n \to \infty$ with $n \equiv 1 \mod d$}
\]
with $\mC_n^\bullet \in \cC_n^\bullet$ drawn uniformly at random. We call $H$ the \emph{expected rescaled height}. Moreover, by applying Corollary~\ref{co:sizepoly} we may assume that
\[
|\cC_n| \sim{} c\, n^{-5/2} \rho^{-n}_Cn! \quad \text{as $n \to \infty$ with $n \equiv 1 \mod d$}
\]
with $c=yd/\sqrt{2 \pi \sigma^2}$. In this section we derive analytical expressions for the relevant constants $\kappa, H, c, \rho,y,\lambda, \sigma^2$ for several graph classes; Table~\ref{tb:constants} provides numerical approximations. For a set of graphs $M$, we denote by $\Forb(M)$ the class of all connected graphs that contain none of the graphs in $M$ as a topological minor; if $M$ contains only 2-connected graphs, then it is easy to see that $\Forb(M)$ is block-stable, cf.\ \cite{MR2744811}. For $n \ge 3$ we denote by $C_n$ a graph that is isomorphic to a cycle with $n$ vertices.

{
\small
\setlength{\extrarowheight}{2.5 pt}

\begin{table}
\begin{center}
    \begin{tabular}{ |l| l | l | l | l | l | l | l | }
    \hline
    Graph Class &$\sca$ & $H$& $c$ & $\rhoc$ & $y$ & $\lambda$ & $\sigma^2$\\ 		\hline
Trees = $\Forb(C_3)$ & 1& $2.50662$ & $0.39894$&$0.36787$&$1$&$1$&$1$ \\
    \hline
     $\Forb(C_4)$ &$1$&$2.13226$&$0.20973$&$0.23618$&$0.27520$&$0.80901$&$1.38196$ \\
    \hline
    $\Forb(C_5)$ &$1.10355$&$1.88657$&$0.10987$&$0.06290$&$0.40384$&$1.85945$&$2.14989$ \\
    \hline
    Cacti Graphs &$1.20297$&$1.99021$& $0.12014$ & $0.23874$ & $0.45631$ & $0.64779$& $2.29559$\\
    \hline
    Outerplanar Graphs &$5.08418$&$1.30501$& $0.00697$ & $0.13659$ & $0.17076$ & $0.22327$ & $95.3658$\\
    \hline
   	\end{tabular}
	\end{center}
	\caption{Numerical approximations of constants for examples of subcritical classes of connected graphs.}
	\label{tb:constants}
	\end{table}
}

\subsection{Trees}
Let $\cC$ be the class of trees, i.e.\ $\cB$ consists only of the graph $K_2$. It is easy to see that the offspring distribution follows a Poisson distribution with parameter one. We immediately obtain:
\begin{proposition}
For the class of tress we have $\kappa=1$ and $\sigma^2=1$.
\end{proposition}
 
\subsection{$\Forb(C_4)$}
Let $\cC$ denote the connected graphs of the class $\Forb(C_4)$. Then each block is either isomorphic to $K_2$ or $K_3$. Hence $B(z) = z^2/2 + z^3/6$. Moreover, for any $B\in\cB$ and any two distinct vertices in $B$ their distance is one. A simple computation then yields:
\begin{proposition}
For the class $\Forb(C_4)$ we have $\kappa=1$ and $\sigma^2= (5 - \sqrt{5})/2$.
\end{proposition}

\subsection{$\Forb(C_5)$}
Recall that the class $\Forb(C_5)$ consists of all graphs that do not contain a cycle with five vertices as a topological minor. Hence, a graph belongs to this class if and only if it contains no cycle of length at least five as subgraph. 
\begin{proposition}
\label{pro:forbc5}
For the class $\Forb(C_5)$ the constant $y$ is the unique positive solution to $zB''(z) = 1$, where $B'$ is given in~\eqref{eq:B'C5}. Moreover, we have
\[
\sca =  \left( 2\,{y}^{2}+4\,y+3 \right) y e^y -\left( 3\,{y}^{2}+12\,y+4 \right) y/2 \approx 1.10355.
\]
and $\sigma^2 = 1 + B'''(y)y^2 \approx 2.14989$.
\end{proposition}
Before proving Proposition~\ref{pro:forbc5} we identify the unlabeled blocks of this class. This result (among extensions to $\Forb(C_6)$ and $\Forb(C_7)$) was given by  Gim\'enez, Mitsche and Noy \cite{2013arXiv1304.5049G}.
\begin{proposition}
The unlabeled blocks of the class $\Forb(C_5)$ are given by
\begin{align}
\label{eq:exc5bic}
K_2, K_4, (K_{2,m})_{m \ge 1},  (K^+_{2,m})_{m \ge 2}.
\end{align}
Here $K_n$ denotes the complete graph and $K_{m, n}$ the complete bipartite graph with bipartition $\{A_m, B_n \}$. The graph $K^+_{2,n}$ is obtained from $K_{2, n}$ by adding an additional edge between the two vertices from $A_2$. 
\end{proposition}
\begin{proof}
We may verify~\eqref{eq:exc5bic} by considering the standard decomposition of 2-connected graphs: an arbitrary graph $G$ is 2-connected if and only if it can be constructed from a cycle by adding $H$-paths to already constructed graphs $H$ \cite{MR2744811}. If $G \in\Forb(C_5)$, then so do all the graphs along its decomposition. In particular we must start with a triangle or a square. Since every edge of a 2-connected graph lies on a cycle, we may only add paths of length at most two in each step. In particular, for $m \ge 3$ a $K_{2,m}$ may only become a $K^+_{2,m}$ or $K_{2,m+1}$, and a $K^+_{2,m}$ may only become a $K^+_{2,m+1}$. Thus~(\ref{eq:exc5bic}) follows by induction on the number of vertices.
\end{proof}

\begin{proof}[Proof of Proposition~\ref{pro:forbc5}.]
\noindent With foresight, we use the decomposition
\begin{align}
\cB = \cS + \cH + \cP
\end{align}
with the classes of labeled graphs $\cS$, $\cH$ and $\cP$ defined by their sets of unlabeled graphs $\tilde{\cS} = \{ K_2, K_3, K_4, C_4\}$, $\tilde{\cH} = \{K_{2,m} \mid m \ge 3\}$ and $\tilde{\cP} = \{K^+_{2,m} \mid m \ge 2\}$.
Any unlabeled graph from $\cH$ or $\cP$ with $n$ vertices has exactly $\binom{n}{2}$ different labelings, since any labeling is determined by the choice of the two unique vertices with degree at least three. Hence
\[
S(x) = x^2/2 + x^3/6 + x^4/6, \quad H(x) = \sum_{n \ge 5} \binom{n}{2} \frac{x^n}{n!} \quad \text{and} \quad P(x) = \sum_{n \ge 4} \binom{n}{2} \frac{x^n}{n!}
\]
and thus
\begin{equation}
\label{eq:B'C5}
B'(x) = x(x+2) {{\rm e}^{x}}-x \left( 15\,x+2\,{x}^{2}+6 \right) / 6 .
\end{equation}
Solving the equation $B'^\bullet(y)=1$ yields
\[
y \approx 0.40384.
\]
First, let $\mH_n \in \cH'^\bullet_n$ with $n \ge 4$ be drawn uniformly at random. We say that a vertex lies on the left if it has degree at least three, otherwise we say it lies on the right. There are $n \binom{n+1}{2}$ graphs in the class $\cH'^\bullet_n$ and precisely $n^2$ of those have the property that the $*$-vertex lies on the left. The distance of the root and the $*$-vertex equals two if they lie on the same side and one otherwise. Hence
\[
	\Ex{\shp(\mH_n)} = \frac{n}{\binom{n+1}{2}} \left( \frac{1}{n} \cdot 2 + \frac{n-1}{n} \cdot1 \right) + \left(1 - \frac{n}{\binom{n+1}{2}} \right) \left( \frac{2}{n}\cdot1 + \frac{n-2}{n} \cdot2 \right).
\]
Let  $\mP_n \in \cP'^{\bullet}_n$ with $n \ge 3$ and $\mS_n \in \cS_n$ with $n=1,2,3$ be drawn uniformly at random. Analogously to the above calculation we obtain
\[
\Ex{\shp(\mP_n)} = \frac{n}{\binom{n+1}{2}}1 + \left(1 - \frac{n}{\binom{n+1}{2}} \right)\left( \frac{2}{n}\cdot1 + \frac{n-2}{n}\cdot2 \right)
\]
and
\[
\Ex{\shp(\mS_1)} = \Ex{\shp(\mS_2)} = 1, \qquad \Ex{\shp(\mS_3)} = \frac{1}{4} \cdot1 + \frac{3}{4}\left(\frac{2}{3}\cdot1 + \frac{1}{3}\cdot2 \right) = \frac{5}{4}.
\]
Since $B'^\bullet(y)=1$ we have for any class $\cF \in \{\cS'^{\bullet}, \cH'^{\bullet},  \cP'^{\bullet}\}$ that
\[
\Ex{\shp(\Gamma B'^\bullet(y)), \Gamma B'^\bullet(y) \in \cF} = \sum_n \left( [z^n]F(yz) \right) \Ex{\shp(\mF_n)}.
\]
Summing up yields
\[
\Ex{\shp(\Gamma B'^\bullet(y))} = \left( 2\,{y}^{2}+4\,y+3 \right) y e^y - \left( 3\,{y}^{2}+12\,y+4 \right) y/2 \approx 1.10355.
\]
\end{proof}

\subsection{Cacti Graphs}
A cactus graph is a graph in which each edge is contained in at most one cycle. Equivalently, the class of cacti graphs is the block-stable class of graphs where every block is either an edge or a cycle. In the following $\cC$ denotes the class of cacti graphs.
\begin{proposition}
For the class of cacti graphs the constant $y$ is the unique positive solution to $zB''(z) = 1$, where $B'$ is given in~\eqref{eq:B'Cac}. Moreover, we have
\[
\sca = {\frac {{y}^{4}-2\,{y}^{3}+2\,y-2}{ \left( {y}^{2}-2\,y+2 \right) 
 \left( 1+y \right)  \left( y-1 \right) }} \approx 1.20297.
\]
and $\sigma^2 = 1 + B'''(y)y^2 \approx 2.29559$.
\end{proposition}
\begin{proof}
\noindent By counting the number of labelings of a cycle, we obtain $|\cB'_n| = n!/2$ for $n \ge 2$. It follows that
\begin{equation}
\label{eq:B'Cac}
B'(z) = z + \frac{z^2}{2(1-z)}
\end{equation}
and hence
$
B'^\bullet(z) = z + \frac{1}{2} \sum_{n\ge 2}n z^n =  \frac{z^3 -2z^2 +2z}{2 (z-1)^2}.
$
Solving the equation $B'^\bullet(y)=1$ yields
\[
	y = -\frac13(17+3\,\sqrt {33})^{1/3}+ \frac23(17+3\,\sqrt {33})^{-1/3}+\frac43 \approx 0.45631.
\]
Let $\Gamma B'^\bullet(y)$ denote a Boltzmann-sampler for the class $\cB'^\bullet$ with parameter $y$ and for any $n \ge 1$ let $\mB_n \in \cB_n'^\bullet$ be drawn uniformly at random. Since $B'^\bullet(y)=1$, it follows that
\[
\sca = \Ex{ \Ex{\shp(\Gamma B'^\bullet(y))\mid |\Gamma B'^\bullet(y)|}}
= \sum_{n \ge 1} \shp(\mB_n) [z^n]B'^\bullet(yz) = \shp(B_1)y + \frac{1}{2} \sum_{n \ge 2} \shp(\mB_n) n y^n.
\]
Clearly $\shp(\mB_1) = 1$ and for $n \ge2 $ we have that $\shp(\mB_n)$ is distributed like the distance from the $*$-vertex to a uniformly at random chosen root from $[n]$ in the cycle $(*,1,2,\ldots,n)$. Hence
\[
\shp(\mB_n) = \begin{cases}
	\frac{2}{n} \sum_{i=1}^{n/2}i = \frac{n+2}{4}, &n \text{ is even}\\
	\frac{n+1}{2n} + \frac{2}{n} \sum_{i=1}^{(n-1)/2}i = \frac{(n+1)^2}{4n}, &n \text{ is odd}
	\end{cases}.
\]
Summing up over all possible values of $n$ yields the claimed expression for $\sca$.
\end{proof}

\subsection{Outerplanar Graphs}
An outerplanar graph is a planar graph that can be embedded in the plane in such a way that every vertex lies on the boundary of the outer face. Let $\cC$ denote the class of connected outerplanar graphs and $\cB$ the subclass consisting of single edges or 2-connected outerplanar graphs.

\begin{proposition}
\label{pr:scouterplanar}
For the class of outerplanar graphs the constant $y$ is the unique positive solution to $zB''(z) = 1$, where $B'(z) = (z + D(z))/2$ and $D$ is given in~\eqref{eq:egfDD}. Moreover, 
\[
\sca = \frac{y}{2} + \left(1-\frac{y}{2}\right)\frac {8\,{\w}^{4}-16\,{\w}^{3}+4\,\w-1}{ \left( 4\,{
\w}^{3}-6\,{\w}^{2}-2\,\w+1 \right)  \left( 2\,\w-1
 \right) } \approx 5.0841
 \text{ with } \w = D(y)
\] 
and $\sigma^2 = 1 + B'''(y)y^2 \approx 95.3658$.
\end{proposition}

Following \cite{MR2789731} we develop a specification of $\cB'^\bullet$ that eventually will enable us to prove the above expressions of the relevant constants. Any 2-connected outerplanar graph has a unique Hamilton cycle, which corresponds  to the boundary of the outer face in any drawing having the property that all vertices lie on the outer face. The edge set of a 2-connected outerplanar graph can thus be partitioned in two parts: the edges of the Hamilton cycle, and all other edges, which we refer to as the set of \emph{chords}. Let $\cBa$ denote the class obtained from $\cB'$ by orienting the Hamilton cycle of each object of size at least three in one of the two directions and marking the oriented edge whose tail is the $*$-vertex. The block consisting of a single edge is oriented in the unique way such that the $*$-vertex is the tail of the marked edge. We start with some observations.

\begin{figure}[b]
	\centering
	\begin{minipage}{0.7\textwidth}
  		\centering
  		\includegraphics[width=1.0\textwidth]{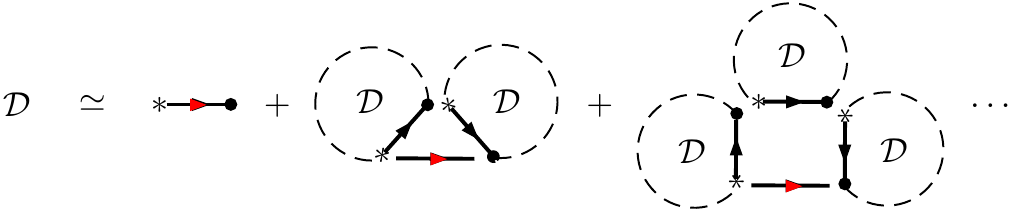}
  		\caption{Recursive specification of the class $\protect \cBa$.}
  		\label{fi:decomp1}
	\end{minipage}
\end{figure}

\begin{lemma}
\label{lem:eqspOP}
We have that $B'(z) = (z + D(z))/2$ and 
\[
\Ex{\shp(\Gamma B'^\bullet(x)} = \frac{x}{2B'^{\bullet}(x)} + (1-\frac{x}{2B'^{\bullet}(x)}) \Ex{\shp(\Gamma \Bra(x))}.
\]
\end{lemma}
\begin{proof}
We have an isomorphism
$
\cB' + \cB' =: 2 \cB' \simeq \cX + \cBa.
$
Consequently, the classes $\cB'^\bullet$ and $\cBra$ obtained by additionally rooting at a non-$*$-vertex satisfy
\[
2 \cB'^\bullet \simeq \cX + \Bra.
\]
Hence the following procedure is a Boltzmann sampler for the class $\cB'^\bullet$ with parameter $x$.

\smallskip
\begin{tabular}{lll}
$\Gamma B'^{\bullet}(x)$:  
		& $s \leftarrow \text{Bern}(\frac{x}{2 B'^{\bullet}(x)})$\\
		& \IF $s = 1$ \THEN \RETURN a single edge $\{*, 1\}$ rooted at $1$\\
		& \ELSE \RETURN $\Gamma \Bra(x)$ without the orientation \\
\end{tabular}
\smallskip

\noindent This concludes the proof.
\end{proof}
\noindent Hence it suffices to study the class $\cBra$, see also Figures  \ref{fi:decomp1} and \ref{fi:decomp0}.
\begin{lemma}
\label{lem:specDDb}
The classes $\cBa$ and $\cBra$ satisfy 
\begin{align*}
\cBa &= \cX + {\cBa}\plh \cBa + {\cBa}\plh \cBa \plh \cBa + \ldots \\
\cBra &= \cX + (\cBra \plh \cBa + \cBa \plh \cBra) + (\cBra \plh {\cBa}\plh \cBa + \cBa \plh \cBra \plh \cBa + {\cBa}\plh \cBa \plh \cBra) + \ldots 
\end{align*}
Their exponential generating functions are given by
\begin{equation}
\label{eq:egfDD}
\Bra(z) = \frac {{z} ( {\Ba(z)}-1 ) ^{2}}{2 {\left(\Ba(z)\right)}^{2}-4 {\Ba(z)}
+1} \quad \text{and} \quad {\Ba(z)} = \frac{1}{4}(1 + z - \sqrt{z^2 - 6z + 1}).
\end{equation}
\end{lemma}

\begin{figure}[t]
	\centering
	\begin{minipage}{1\textwidth}
		\centering
  		\includegraphics[width=0.6\textwidth]{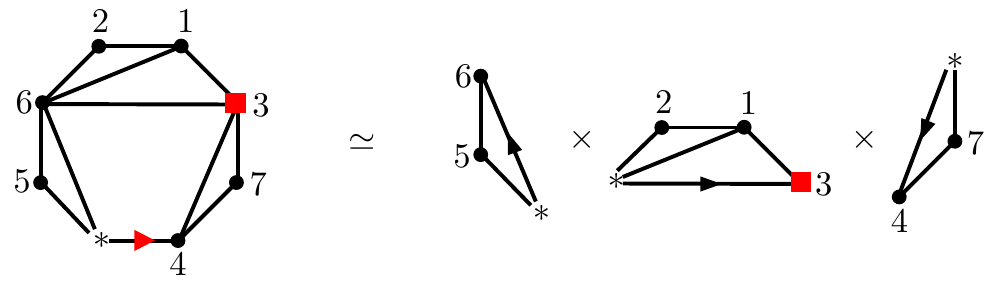}

  		\caption{Decomposition of a $\protect \cBra$-object into a $\protect \cBa \times \protect \cBra \times \protect \cBa$-object. The root is marked with a square.}
   		\label{fi:decomp0}
   	\end{minipage}
\end{figure}

\begin{proof}
Let $B \in \cBa$ with $|B| \ge 2$ be a derived outerplanar block, rooted at an oriented edge $\overrightarrow{e}$ of its Hamilton cycle $C$ such that the $*$-vertex is the tail of $\overrightarrow{e}$.
Given a drawing of $B$ such that $C$ is the boundary of the outer face, the root face is defined to be the bounded face~$F$ whose border contains $\overrightarrow{e}$. Then $B$ may be identified with the sequence of blocks along $F$, ordered in the reverse direction of the edge $\overrightarrow{e}$. This yields the decompositions illustrated in Figures  \ref{fi:decomp1} and \ref{fi:decomp0}. Solving the corresponding equations of generating functions yields~\eqref{eq:egfDD}.
\end{proof}
\noindent The equation determining the $y = C^\bullet(\rhoc)$ is 
$ 1 = B'^\bullet(y) = \frac{1}{2}(y + \Bra(y))$. We obtain that 
$y \approx 0.17076$ is the unique root of the polynomial
$
3\,{z}^{4}-28\,{z}^{3}+70\,{z}^{2}-58\,z+8
$
in the interval $[0,\frac{1}{2}]$ and hence $\sigma^2 = 1 + B'''(y)y^2 \approx 95.3658$. It remains to compute $\sca$.

\begin{lemma}
\label{le:out3}
We have that $\Ex{\shp(\Gamma \Bra(y))} = {\frac {8{\w}^{4}-16{\w}^{3}+4\w-1}{ \left( 4{
\w}^{3}-6{\w}^{2}-2\w+1 \right)  \left( 2\w-1
 \right) }} \approx 5.46545$ with $\w := \Ba(y) \approx 0.27578$.
\end{lemma}
Since $B'^\bullet(y)=1$ this implies with Lemma~\ref{lem:eqspOP} that
\[ \sca = \Ex{\shp(\Gamma \cB'^\bullet(y))} = \frac{y}{2} + \left(1-\frac{y}{2}\right) \Ex{\shp(\Gamma \Bra(y))} \approx 5.08418,\]
and the this completes the proof of Proposition~\ref{pr:scouterplanar}.
\begin{proof}[Proof of Lemma \ref{le:out3}]
The rules for Boltzmann samplers translate the specification of $\cBra$ given in Lemma~\ref{lem:specDDb} into the following sampling algorithm.

\medskip
\begin{tabular}{lll}
$\Gamma \Bra(x)$:  
		& $s \leftarrow $ drawn with $\Pr{s=2} = \frac{x}{\Bra(x)}$ and $\Pr{s=i} = (i-1){\left(\Ba(x)\right)}^{i-2}$ for $i \ge 3$ \\
		& \IF $s = 2$ \THEN\\
		& \qquad \RETURN a single directed edge $(*, 1)$\\
		& \ELSE \\
		& \qquad $\gamma \leftarrow $ a cycle $\{v_1, v_2\}, \{v_2, v_3\}, \ldots, \{v_{s},v_1\}$ with $v_1 = *$\\
		& \qquad $t \leftarrow $ a number drawn uniformly at random from the set $[s-1]$ \\
		& \qquad $\gamma \leftarrow$ identify $(v_t, v_{t+1})$ with the root-edge of $\gamma_t \leftarrow \Gamma \Bra(x)$ \\
		& \qquad \FOR \EACH $i \in [s-1] \setminus \{t\}$\\
		& \qquad \qquad $\gamma \leftarrow$ identify $(v_i, v_{i+1})$ with the root-edge of $\gamma_i \leftarrow \Gamma \Ba(x)$ \\
		& \qquad \SPENDFOR \\
		& \qquad root $\gamma$ at the directed edge $(*, v_{s})$ \\
		& \qquad \RETURN $\gamma$ relabeled uniformly at random\\
		& \ENDIF
	\end{tabular}
\medskip

\noindent Given a graph $H$ in $\cBra$ let $S(H)$, $S'(H)$ denote the length of a shorted past in $H$ from the root-vertex to the tail $v_1 = *$ or head $v_{s}$ of the directed root-edge, respectively. Clearly, $S(H)$ and $S'(H)$ differ by at most one. It will be convenient to also consider their minimum $M(H)$. Let $S$, $S'$ and $M'$ denote the corresponding random variables in the random graph $\mathsf{D}$ drawn according to the sampler $\Gamma \Bra(x)$. For any integers $\ell,k \ge 0$ with $\ell + k \ge 1$ let~$\mathsf{D}_{\ell,k}$ be the random graph $\mathsf{D}$ conditioned on the event that the graph is not a single edge and that in the root face  $\{v_1,v_2\}, \{v_2, v_3\} \ldots, \{v_{s},v_1\}$ the length of the path $v_1v_2 \ldots v_t$ equals~$\ell$ and the length of the path $v_{t+1} v_{t+2} \ldots v_{s}$ equals $k$. Note that the probability for this event equals
\[
	p_{\ell,k} = \Pr{s=\ell+k+2} \Pr{t=\ell+1 \mid s=\ell+k+2} = {\left(\Ba(x)\right)}^{k+\ell}.
\]
We denote by $S_{\ell,k}$, $S'_{\ell,k}$ and $M_{\ell,k}$ the corresponding distances in the conditioned random graph $\mathsf{D}_{\ell,k}$. Summing over all possible values for $k$ and $\ell$ we obtain
\begin{align*}
\Ex{S} = \frac{x}{\Bra(x)} + \sum_{k+\ell \ge 1} \Ex{S_{\ell,k}} p_{\ell,k}, \,\quad
\Ex{S'} = \sum_{k+\ell \ge 1} \Ex{S'_{\ell,k}} p_{\ell,k}, \,\quad
\Ex{M} = \sum_{k+\ell \ge 1} \Ex{M_{\ell,k}} p_{\ell,k}.
\end{align*}
Any shortest path from $*$ or $v_s$ to the root-vertex of a $\cB'^\bullet$-graph $H$ ($\ne$ a single edge) must traverse the boundary of the root-face in one of the two directions until it reaches the root-edge of the attached $\cB'^\bullet$-object $H'$. From there it follows a shortest path to the root in the graph $H'$. Hence for all $k,\ell \ge 0$ with $k+\ell \ge 1$ the following equations hold.
\begin{align*}
S_{\ell,k} &\eqdist \min\{\ell+S, k+1+S'\}, \\
S'_{\ell,k} &\eqdist \min\{\ell+1+S, k+S'\}, \\
M_{\ell,k} &\eqdist \min\{\ell+S, k+S'\}.
\end{align*}
Since $S$ and $S'$ differ by at most one, this may be simplified further depending on the parameters $k$ and $\ell$ as follows:
\begin{alignat*}{6}
S_{\ell,k} &\eqdist \begin{cases}
	\ell + S, &\ell \le k\\
	\ell + M, &\ell = k+1 \\
	k+1 + S', &\ell \ge k+2
	\end{cases}, \quad& 
	S'_{\ell,k} \eqdist \begin{cases}
	k + S', &k \le \ell\\
	k + M, &k = \ell+1 \\
	\ell+1 + S, &k \ge \ell+2
	\end{cases},
\end{alignat*}
and
\begin{alignat*}{6}
	M_{\ell,k} &\eqdist \begin{cases}
	\ell + S, &\ell \le k-1\\
	\ell + M, &\ell = k \\
	k + S', &\ell \ge k+1
	\end{cases}.
\end{alignat*}
Using this and~\eqref{eq:egfDD}, we arrive at the system of linear equations with parameter $\w = \Ba(x)$ and variables $\mu_{S} = \Ex{S}$, $\mu_{S'}=\Ex{S'}$ and $\mu_{M}=\Ex{M}$
\[
\begin{split}
	\mu_S
		&= {\frac {2{\w}^{2}-4\w+1}{ ( \w-1 ) ^{2}}} + \sum_{k\ge1}\sum_{\ell=0}^{k}(\ell + \mu_S)\w^{\ell+k} +  \sum_{\ell\ge1}(\ell + \mu_M) \w^{2\ell-1} +\sum_{k\ge 0}\sum_{\ell \ge k+2}(k+1+\mu_{S'})\w^{\ell+k},	\\
	\mu_{S'}
		&= \sum_{\ell \ge1}\sum_{k=0}^\ell(k + \mu_{S'})\w^{\ell+k} + \sum_{k\ge1}(k + \mu_{M})\w^{2k-1} + \sum_{\ell \ge 0} \sum_{k\ge\ell+2}(\ell+1+\mu_{S})\w^{\ell+k}, \\
	\mu_{M}
		&= \sum_{k\ge2} \sum_{\ell=0}^{k-1}(\ell+\mu_{S})\w^{\ell+k} + \sum_{\ell\ge1}(\ell+\mu_{M})\w^{2\ell} +\sum_{k\ge0} \sum_{\ell\ge k+1} (k+\mu_{S'})\w^{\ell+k}.
\end{split}
\]
Simplifying the equations yields the equivalent system $A \cdot (\Ex{S}, \Ex{S'}, \Ex{M})^\top = b$ with
\begin{align*}
A = 
\begin{pmatrix}
2{\w}^{4} -4{\w}^{3}+3\w-1  & -{\w}^{3}+{\w}^{2} & {\w}^{3}-2{
\w}^{2}+\w  \\
-{\w
}^{3}+{\w}^{2} & 2{\w}^{4}-4{\w}^{3}+
3\w-1 & {\w}^{3}-2{\w}^{2}+\w\\
-{\w}^{2}+\w & -{\w}^{2}+\w & 2{\w}^{4}-4{\w}^{3}+{\w}^{2}+2 \w-1 
\end{pmatrix}
\end{align*}
and
\begin{align*}
b^\top = 
\begin{pmatrix}
2{\w}^{4}-4{\w}^{3}-{\w}^
{2}+3\w-1 & -\w & -{\w}^{2}
\end{pmatrix}.
\end{align*}
\noindent For $x=y\approx 0.17076$ we obtain $w \approx 0.27578$ and $\text{det}(A) \approx -0.00799 \ne 0$. Solving the system of linear equations yields
\begin{align*}
\Ex{S}&={\frac {8{\w}^{4}-16{\w}^{3}+4\w-1}{ \left( 4{
\w}^{3}-6{\w}^{2}-2\w+1 \right)  \left( 2\w-1
 \right) }} \approx 5.46545, \\
\end{align*}
and the proof is completed.
\end{proof}

\section*{Acknowledgements}
We would like to express our thanks to Gr\'egory Miermont and Igor Kortchemski for helpful suggestions and references regarding the Continuum Random Tree. The  second author would like to thank his wife for her support during the writing of this paper.

\bibliographystyle{abbrv}
\bibliography{sclimit}

\end{document}